\documentclass[smallextended]{svjour3}       % onecolumn (second format)
\smartqed  % flush right qed marks, e.g. at end of proof

\usepackage{setspace}
\usepackage[active]{srcltx}
\usepackage{pict2e}
\usepackage{graphicx}
\usepackage{amsmath}
\usepackage{amsfonts}
\usepackage{amssymb}
\usepackage{float}
\usepackage{multirow}
\usepackage{epsfig}
\usepackage{subfigure}
\usepackage{psfrag}
\usepackage{booktabs}
\usepackage{algorithm}
\usepackage{algpseudocode}

\numberwithin{equation}{section}
\usepackage[pagebackref=true,breaklinks=true,letterpaper=true,colorlinks,bookmarks=false]{hyperref}

\usepackage[normalem]{ulem} % to use \sout

\newcommand{\vb}{{\mathbf{b}}}

\newcommand{\vp}{{\mathbf{p}}}

\newcommand{\vu}{{\mathbf{u}}}
\newcommand{\vv}{{\mathbf{v}}}
\newcommand{\vw}{{\mathbf{w}}}
\newcommand{\vx}{{\mathbf{x}}}
\newcommand{\vy}{{\mathbf{y}}}

\newcommand{\vA}{{\mathbf{A}}}

\newcommand{\vI}{{\mathbf{I}}}
\newcommand{\vJ}{{\mathbf{J}}}

\newcommand{\vP}{{\mathbf{P}}}

\newcommand{\vR}{{\mathbf{R}}}

\newcommand{\cC}{{\mathcal{C}}}

\newcommand{\cL}{{\mathcal{L}}}

\newcommand{\cP}{{\mathcal{P}}}

\newcommand{\cV}{{\mathcal{V}}}

\newcommand{\sign}{\mathrm{sign}}
\newcommand{\vzero}{\mathbf{0}}

\DeclareMathOperator*{\argmin}{arg\,min}

\newcommand{\bc}{\begin{center}}
\newcommand{\ec}{\end{center}}

\newcommand{\bdm}{\begin{displaymath}}
\newcommand{\edm}{\end{displaymath}}

\newcommand{\beq}{\begin{equation}}
\newcommand{\eeq}{\end{equation}}

\newcommand{\bfl}{\begin{flushleft}}
\newcommand{\efl}{\end{flushleft}}

\newcommand{\bt}{\begin{tabbing}}
\newcommand{\et}{\end{tabbing}}

\newcommand{\beqn}{\begin{eqnarray}}
\newcommand{\eeqn}{\end{eqnarray}}

\newcommand{\beqs}{\begin{align*}} % no equation numbers
\newcommand{\eeqs}{\end{align*}}  % no equation numbers

\newcommand{\prox}{{\bf prox}}

\DeclareMathOperator*{\Min}{minimize}

\begin{document}
\title{Nonconvex Sorted $\ell_1$ Minimization for Sparse Approximation\thanks{This work is partially supported by ERC AdG A-DATADRIVE-B, IUAP-DYSCO, GOA-MANET and OPTEC, National Natural Science Foundation of China (11201079), the Fundamental Research Funds for the Central Universities of China (20520133238, 20520131169), and NSF grants DMS-0748839 and DMS-1317602.}}

\titlerunning{Nonconvex Sorted $\ell_1$ Minimization}        % if too long for running head

\author{Xiaolin Huang \and  Lei Shi \and Ming Yan}


\institute{
X. Huang \at Department of Electrical Engineering, KU Leuven, B-3001 Leuven, Belgium. \\\email{huangxl06@mails.tsinghua.edu.cn}\and
L. Shi \at  Shanghai Key Laboratory for Contemporary Applied Mathematics, School of Mathematical Sciences, Fudan University, Shanghai, 200433, P.R. China. \\\email{leishi@fudan.edu.cn} \and
M. Yan \at Department of Mathematics, University of California, Los Angeles, CA 90095, USA. \\
              \email{yanm@math.ucla.edu}
							}

\date{Received: date / Accepted: date}

\maketitle

\begin{abstract}
The $\ell_1$ norm is the tight convex relaxation for the $\ell_0$ ``norm" and has been successfully applied for recovering sparse signals. For problems with fewer samplings, one needs to enhance the sparsity by nonconvex penalties such as $\ell_p$ ``norm". As one method for solving $\ell_p$ minimization problems, iteratively reweighted $\ell_1$ minimization updates the weight for each component based on the value of the same component at the previous iteration. It assigns large weights on small components in magnitude and small weights on large components in magnitude. In this paper, we consider a weighted $\ell_1$ penalty with the set of the weights fixed and the weights are assigned based on the sort of all the components in magnitude. The smallest weight is assigned to the largest component in magnitude. This new penalty is called nonconvex sorted $\ell_1$. Then we propose two methods for solving nonconvex sorted $\ell_1$ minimization problems: iteratively reweighted $\ell_1$ minimization and iterative sorted thresholding, and prove that both methods will converge to a local optimum. We also show that both methods are generalizations of iterative support detection and iterative hard thresholding respectively. The numerical experiments demonstrate the better performance of assigning weights by sort compared to $\ell_p$ minimization.

\bigskip

\keywords{Iteratively Reweighted $\ell_1$ Minimization, Iterative Sorted Thresholding, Local Minimizer, Nonconvex Optimization, Sparse Approximation.}
\subclass{49M37 \and 65K10  \and 90C26  \and 90C52}
\end{abstract}

%%% Introduction %%%
\section{Introduction}\label{sec:intro}

Sparse approximation aims at recovering a sparse vector $\vu\in\vR^n$ from relatively few linear measurements $\vb\in\vR^m$. More precisely, we determine a sparse vector $\vu\in\vR^n$ from an underdetermined linear system
\begin{align*}
\vA\vu=\vb,
\end{align*}
where $\vA\in \vR^{m\times n}$ is the measurement matrix with $m<n$. Many problems such as signal and image compression, compressed sensing, and error correcting codes, can be modeled as sparse approximation problems~\cite{bruckstein2009sparse,starck2010sparse}. In statistics, sparse approximation is also related to selecting the relevant explanatory variables, which is referred to as model selection~\cite{buhlmann2011statistics}.

This underdetermined linear system has infinite solutions and we are interested in the sparest solution only. Finding the sparest vector amounts to solving the $\ell_0$ minimization problem
\begin{align}\label{sec:L0-minimization}
\min_{\vu\in \vR^n} \|\vu\|_0, \mbox{ subject to } \vA\vu=\vb,
\end{align}
where $\|\vu\|_0$ is the number of nonzero components in $\vu$.

Since the locations of the nonzero components are not available, solving problem~\eqref{sec:L0-minimization} directly is NP-hard in general~\cite{natarajan1995sparse}. A family of iterative greedy algorithms, including orthogonal matching pursuit~\cite{tropp2007signal}, CoSaMP~\cite{needell2009cosamp}, subspace pursuit~\cite{dai2009subspace}, iterative hard thresholding~\cite{blumensath2009iterative}, and hard thresholding pursuit~\cite{foucart2011hard}, have been established to solve problem~\eqref{sec:L0-minimization} with less computational complexity.

Another alternative is to consider the $\ell_1$ minimization problem
\begin{align}\label{sec:L1-minimization}
\min_{\vu\in \vR^n} \|\vu\|_1, \mbox{ subject to } \vA\vu=\vb.
\end{align}
%The global optimum of~\eqref{sec:L1-minimization} can be easily computed by existing techniques, e.g., the interior-point method and the simplex method.
The convex optimization problem~\eqref{sec:L1-minimization}, commonly known as basis pursuit~\cite{chen1998atomic}, is an efficient relaxation for problem~\eqref{sec:L0-minimization}. The $\ell_1$ minimization often leads to sparse solutions and the mechanism behind its performance has been theoretically analyzed. It is known from the compressed sensing literature that if $\vA$ obeys some conditions, such as the restricted isometry property~\cite{candes2006robust}, the null space property~\cite{cohen2009compressed}, and the incoherence condition~\cite{tropp2004greed}, problem~\eqref{sec:L0-minimization} and its convex relaxation~\eqref{sec:L1-minimization} are equivalent. When there is noise in the measurement, the following basis pursuit denoising model is proposed in~\cite{tibshirani1996regression}:
\begin{align}\label{sec:BPDN}
\min_{\vu\in \vR^n} \|\vu\|_1+\alpha\|\vA\vu-\vb\|^2.
\end{align}

Though $\ell_1$ minimization is stable and has a number of theoretical results, it is not able to recover the sparest solutions in many real applications, e.g., computed tomography, where the sufficient conditions for exact recovery are not satisfied. The noncovex optimization is applied to enhance the sparsity of the solutions and improve the recovery performance. The mostly used non-convex penalty is~$\ell_p$ ``norm" with $0<p<1$~\cite{chartrand2007exact,foucart2009sparsest,saab2010sparse,xu2012regularization}, which connects $\ell_0$ and $\ell_1$. Other nonconvex penalties in the literature are: smoothly clipped absolute deviation (SCAD)~\cite{fan2001variable}, generalized shrinkage~\cite{chartrand2013generalized,chartrand2014shrinkage}, etc. For almost all the nonconvex penalties used, the regularization terms penalize components with small magnitudes more than those with large magnitudes and are separable.

In this paper, we introduce a new nonconvex penalty called nonconvex sorted $\ell_1$ penalty which penalizes different components in the solution according to their ranks in the absolute value. It is not separable and the set of weights is fixed. The contributions of this paper are summarized below.
\begin{itemize}
\item We introduce the nonconvex sorted $\ell_1$ minimization to enhance the sparsity and improve recovery performance.
\item We build the connection of the nonconvex sorted $\ell_1$ to several existing penalties including $\ell_1$ and $K$-sparsity.
\item We propose an iteratively reweighted $\ell_1$ minimization for solving the problems with nonconvex sorted $\ell_1$ terms, and show that it converges in finite steps to a local minimizer.
\item We propose an iterative sorted thresholding method for the unconstrained denoising problem and show that it converges to a local optimum. This iterative sorted thresholding is a generalization of iterative soft thresholding and iterative hard thresholding.
\end{itemize}

\subsection{Notation}
The following notations are used throughout the paper. For any dimension $n$, bold lowercase letters are used for vectors and lowercase letters with subscripts denote their components, e.g., $\vu=(u_1,\cdots,u_n)^T\in\vR^n$. Bold uppercase letters such as $\vA$ and $\vP$ are used for matrices. $\vI_n$ stands for the $n\times n$ identity matrix, and $\vJ_n$ is the all-ones $n\times n$ matrix. For $\vu\in\vR^n$, the $\ell_p$ norm of $\vu$ is $\|\vu\|_p:=(\sum_{i=1}^n|u_i|^p)^{1/p}$ for $0<p<\infty$, with the usual extension $\|\vu\|_0:=\#\{i:u_i\neq 0\}$ and $\|\vu\|_{\infty}:=\max_{1\leq i\leq n}|u_i|$.  Strictly speaking, $\|\cdot\|_0$ is not a real norm and $\|\cdot\|_p$ merely defines a quasi-norm when $0<p<1$. For simplicity, let $\|\cdot\|$ stands for $\|\cdot\|_2$. The indicator function $\iota_{\cC}$ of the set $\cC$ is defined as follows:
\begin{align*}
\iota_\cC(\vx)=\left\{\begin{array}{ll}0,&\mbox{ if }\vx\in\cC,\\+\infty,&\mbox{ if }\vx\notin\cC.\end{array}\right.
\end{align*}
The component-wise multiplication of two vectors $\vx$ and $\vy$ is defined as follows:
\begin{align*}
(\vx\odot\vy)_i=x_i\times y_i,\quad\forall i\in[1,n].
\end{align*}

\subsection{Organization}
The rest of the paper is presented as follows. The previous related works on nonconvex optimization for sparse approximation are given in section~\ref{sec:previous}. We introduce the nonconvex sorted $\ell_1$ minimization in section~\ref{sec:nonconvex_minimization}, and propose two methods for solving the problem with nonconvex sorted $\ell_1$ terms in sections~\ref{sec:alg1} and~\ref{sec:alg2}. The convergence results are shown in the corresponding sections. In section~\ref{sec:numerical}, we compare the performance of this nonconvex sorted $\ell_1$ with other nonconvex approaches on compressed sensing problems for both noise-free and noisy cases. This paper is ended with a short conclusion section.

%%% Previous Works %%%
\section{Previous Works}\label{sec:previous}

As shown in applications of sparse approximation, nonconvex minimizations have better performance than convex optimizations in enforcing sparsity. The $\ell_p$ minimization with $0<p<1$ is mostly used in literature as a bridge between $\ell_0$ and $\ell_1$. It was demonstrated that under certain restricted isometry properties, the $\ell_p$ minimization problem
\begin{align}\label{sec:Lp-minimization}
\min_{\vu\in \vR^n} \|\vu\|^p_p, \mbox{ subject to } \vA\vu=\vb
\end{align} recovers sparse vectors from fewer linear measurements than the $\ell_1$ minimization (\ref{sec:L1-minimization}) does~\cite{chartrand2007exact,foucart2009sparsest}. Because of the nonconvexity, finding a global minimizer of problem~\eqref{sec:Lp-minimization} or its unconstrained variant is generally NP-hard~\cite{ge2011note,Chen2014complexity}.

The first group of methods for solving problem~\eqref{sec:Lp-minimization} is to solve a sequence of convex problems. The iteratively reweighted $\ell_1$ minimization (IRL1)~\cite{candes2008enhancing} is to iteratively solve the following weighted $\ell_1$ minimization
\begin{align}\label{sec:IRL1}
\vu^{l+1}=\arg\min_{\vu\in \vR^n} \left\{\sum_{i=1}^n w^{l}_i |u_i|, \mbox{ subject to } \vA\vu=\vb\right\},
\end{align} where $w^{l}_i=(|u^{l}_i|+\epsilon)^{p-1}$ and $\epsilon>0$. A parallel approach to IRL1 is iteratively reweighted least squares (IRLS)~\cite{chartrand2008iteratively,daubechies2010iteratively}, which iteratively solves the following weighted least square problem,
\begin{align}\label{sec:IRLS}
\vu^{l+1}=\arg\min_{\vu\in \vR^n} \left\{\sum_{i=1}^n w^{l}_i u^2_i, \mbox{ subject to } \vA\vu=\vb\right\},
\end{align} with $w^{l}_i=\left((u^{l}_i)^2+\epsilon\right)^{p/2-1}$ for $i=1,\cdots,n$.
When $\epsilon\to 0$, the objective functions in (\ref{sec:IRL1}) and (\ref{sec:IRLS}) are approximations to $\|\vu\|^p_p$. The $\epsilon$-regularization strategy, which starts with a relatively large $\epsilon$ and decreases $\epsilon$ at each iteration, improves the performance of sparse recovery for both methods. Similar iteratively reweighted algorithms have been established in \cite{lai2013improved,chen2013convergence} for unconstrained smoothed $\ell_p$ minimization. %Moreover, it is showed that under certain conditions, any accumulation point in the solution sequence is a stationary point where the first-order optimality condition holds. However, as an alternative method to solve the $\ell_p$ minimization, whether the solution of iteratively reweighted algorithms can approximate a locally optimal of the original $\ell_p$ minimization is still unclear.

Additionally, iterative thresholding algorithms have been established for unconstrained $\ell_p$ minimization problems~\cite{blumensath2008iterative,xu2012regularization}. These algorithms can search a local minimizer starting from any initial point using the so-called shrinkage operator to generate a minimizing sequence such that the objective function is strictly decreasing along the sequence. The iterative thresholding algorithm is very efficient and particularly suitable for large-scale problems. But except for $p=0,1/2,2/3,1$, one can not give an explicit expression for the shrinkage operator which limits its applications in the $\ell_p$ minimization~\cite{xu2012regularization}.

Except the nonconvex $\ell_p$, other nonconvex penalties have been proposed. The SCAD penalty in statistics~\cite{fan2001variable}, of which the contour map is illustrated in Fig.~\ref{fig-example-concave-losses}, is defined as $R_{\textnormal{SCAD}}(\vu)=\sum_{i=1}^n r_{\textnormal{SCAD}}(u_i)$, where
\begin{align*}
r_{\textnormal{SCAD}}(u_i)=\left\{\begin{array}{ll}a_1|u_i| & \mbox{ if }|u_i|<a_1, \\ -{a_1|u_i|^2-2a_1a_2|u_i|+a_1^3\over 2(a_2-a_1)} & \mbox{ if }a_1\leq |u_i|\leq a_2,\\ {a_1a_2+a_1^2\over 2} & \mbox{ if }|u_i|>a_2.\end{array}\right.
\end{align*}
The $\ell_{1-2}$, i.e., the difference between $\ell_1$ and $\ell_2$ norms, minimization is proposed in~\cite{yin2014minimization}. The difference of convex functions programming can be applied to solve those nonconvex optimization problems~\cite{gasso2009recovering}.

All the penalties mentioned above are separable, i.e., the same penalty function is applied on all the components of the vector, except that $\ell_{1-2}$ has a $\ell_2$ term which couples all the components equally. In addition, there are a group of methods applying different penalty functions on different components based on their ranks in the absolute value. Iterative hard thresholding~\cite{blumensath2009iterative} puts weight $0$ on the first $K$ largest components and $+\infty$ on the other components; Iterative support detection~\cite{wang2010sparse} puts weight $0$ on the first $K$ largest components and $1$ on the other components; The two-level $\ell_1$ ``norm"~\cite{huang2014two} puts a smaller positive weight on the first $K$ largest components and a larger weight on the other components. In this paper, we generalize all these methods using more than two different weights, and propose iteratively reweighted $\ell_1$ minimization and iterative sorted thresholding to solve problems with this generalized penalty.

%\newpage

%%% Section 3 %%%
\section{Nonconvex Sorted $\ell_1$ Minimization}\label{sec:nonconvex_minimization}

This section introduces a new nonconvex minimization named \emph{nonconvex sorted $\ell_1$} minimization. Assume $\vu\in\vR^n$. Let $\lambda$ be a nondecreasing sequence of nonnegative numbers,
\begin{align*}
0\leq \lambda_1\leq \lambda_2\leq\cdots\leq\lambda_n,
\end{align*}
with $\lambda_n>0$. The nonconvex sorted $\ell_1$ is defined as
\begin{align}\label{eqn:nonconvexl1}
R_\lambda(\vu)=\lambda_1|\vu|_{[1]}+\lambda_2|\vu|_{[2]}+\cdots+\lambda_n|\vu|_{[n]},
\end{align}
where $|\vu|_{[1]}\geq|\vu|_{[2]}\geq\cdots\geq|\vu|_{[n]}$ are the absolute values ranked in decreasing order. It is different from the sorted $\ell_1$ norm proposed by Bogdan et.al. in~\cite{Bogdan2013statistical}, where $\lambda$ is a nonincreasing sequence of nonnegative numbers, i.e., higher weights are assigned on components with larger absolute values. The contour map of the non-convex sorted $\ell_1$ is illustrated in Fig.~\ref{fig-example-concave-losses}, along with those of $\ell_1$ norm, SCAD, and $\ell_p$ ``norm".

\begin{figure}[!htb]
  \centering
  \subfigure[]{
    \includegraphics[width=0.4\textwidth]{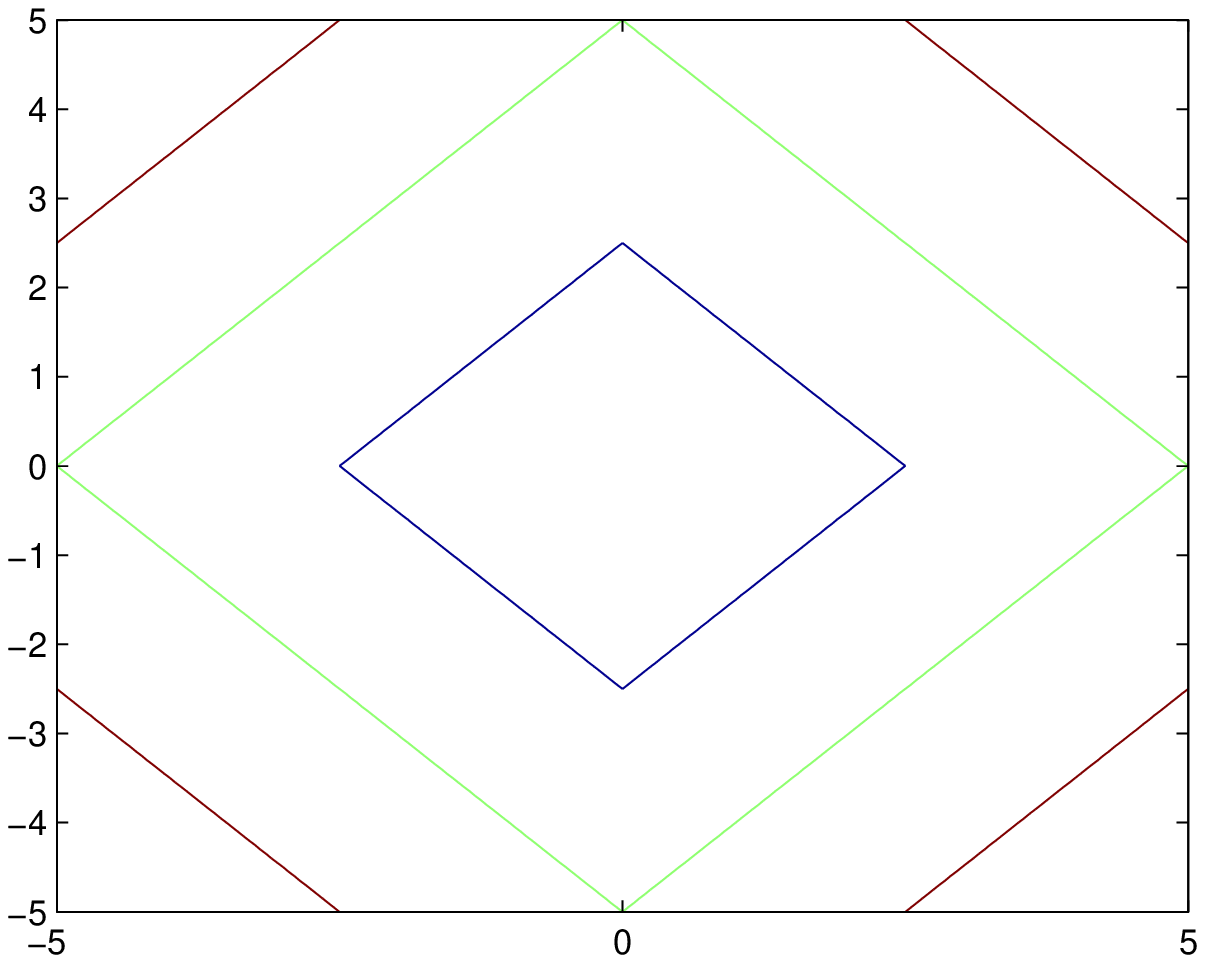}}
  \subfigure[]{
    \includegraphics[width=0.4\textwidth]{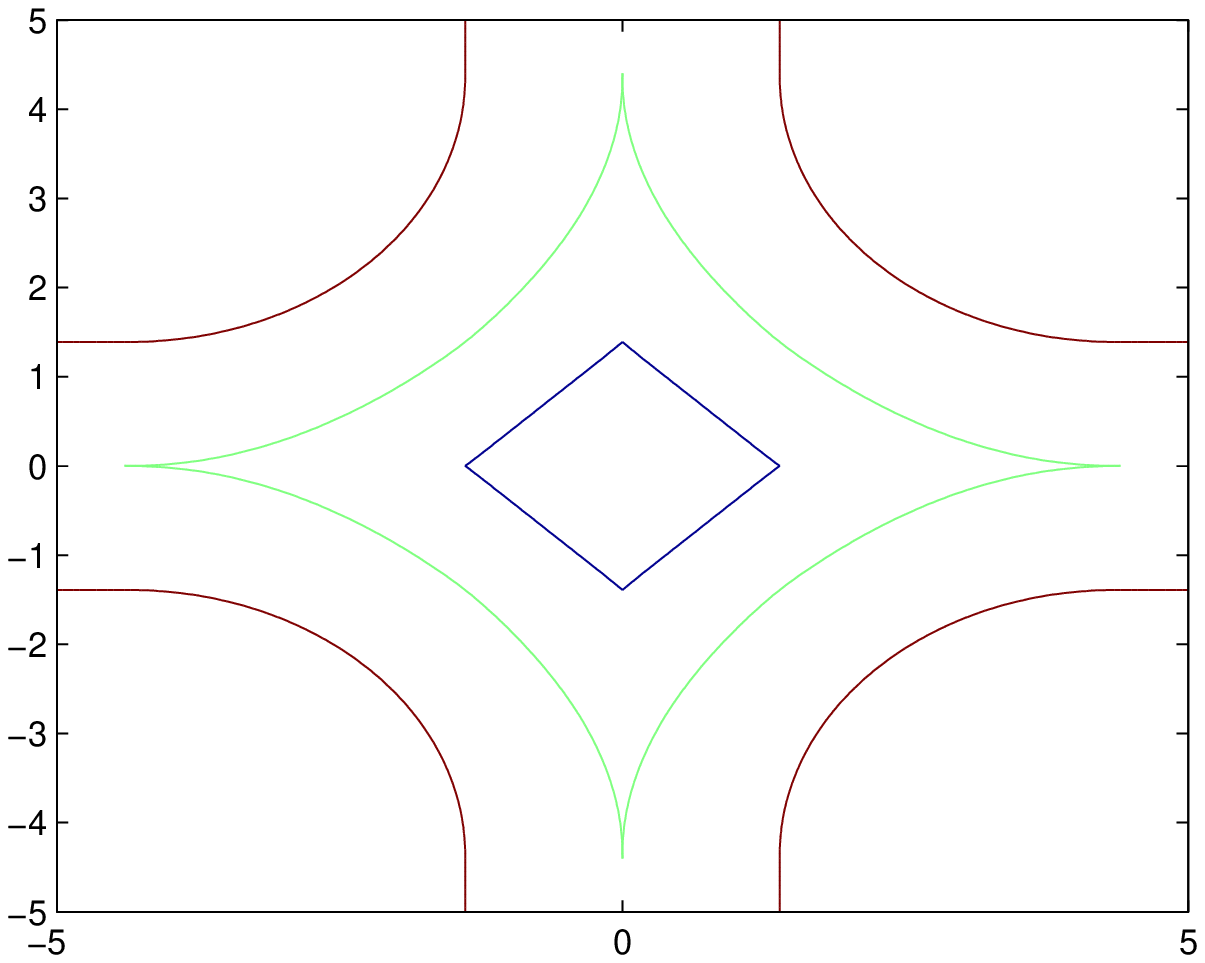}}\\
  \subfigure[]{
    \includegraphics[width=0.4\textwidth]{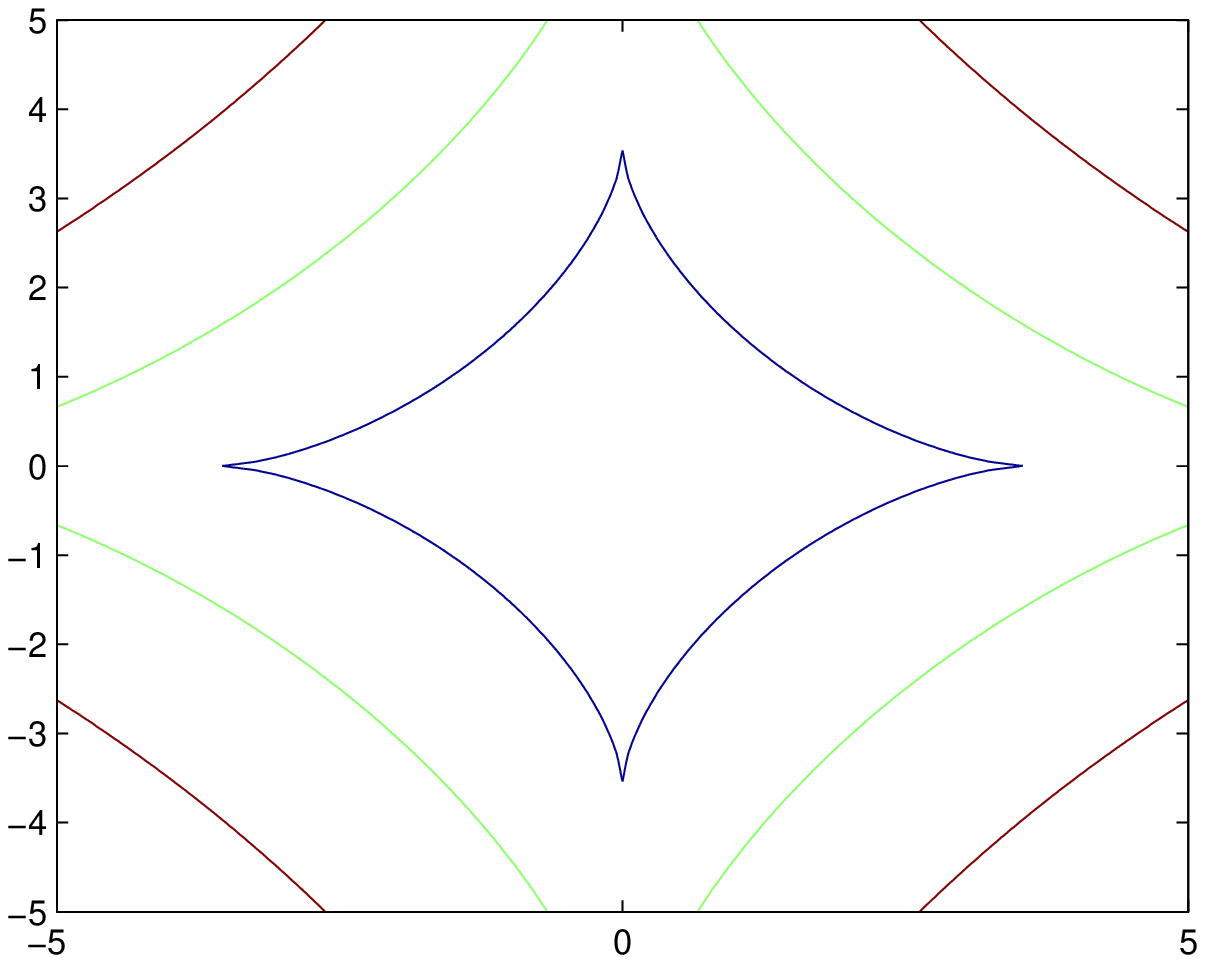}}
  \subfigure[]{
    \includegraphics[width=0.4\textwidth]{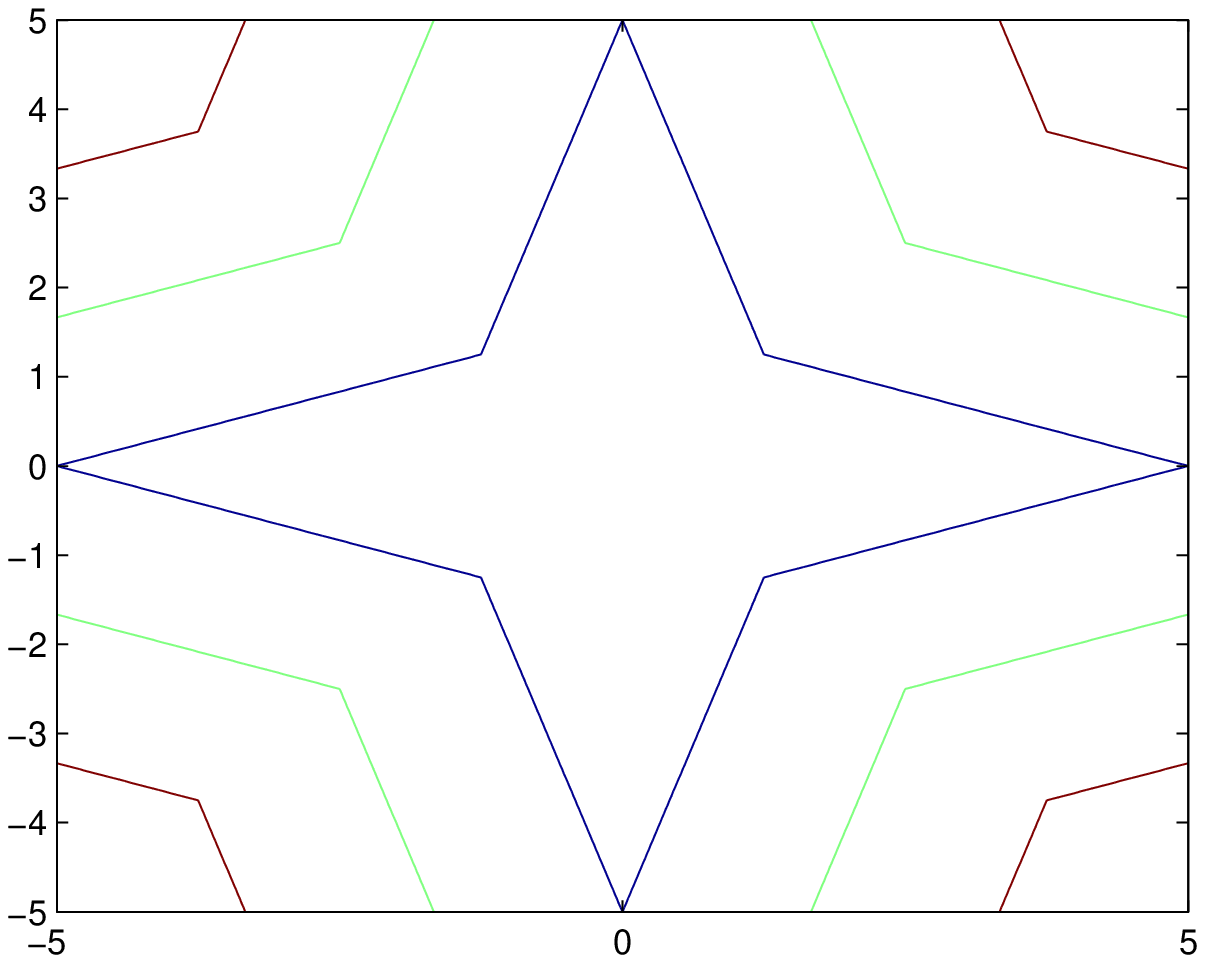}}\\
    \caption{The contour maps of several penalties: (a) the $\ell_1$ norm; (b) the SCAD penalty with $a_1 = 1.1, a_2 = 3.7$; (c) the $\ell_p$-norm with $p = 2/3$; (d) the nonconvex sorted $\ell_1$ with $\lambda_1=1/3$ and $\lambda_2=1$.}\label{fig-example-concave-losses}
\end{figure}

First, we establish the connections of this nonconvex sorted $\ell_1$ minimization to existing works.
\begin{itemize}
\item If $\lambda_1=\lambda_2=\cdots=\lambda_n=1$, it is the $\ell_1$ norm of $\vu$.
\item If $\lambda_1=\cdots=\lambda_K=0$ and $\lambda_{K+1}=\cdots=\lambda_n=+\infty$ for some $K$ satisfying $0<K\leq n$, it is the indicator function for $\{\vu:\|\vu\|_0\leq K\}$.
\item If $\lambda_1=\cdots=\lambda_K=0$ and $\lambda_{K+1}=\cdots=\lambda_n=1$ for some $K$ satisfying $0<K<n$, it corresponds to the iterative support detection in~\cite{wang2010sparse}. In~\cite{wang2010sparse}, $K$ can be changed adaptively during the iterations.
\item If $\lambda_1=\cdots=\lambda_K=w_1$ and $\lambda_{K+1}=\cdots=\lambda_n=w_2$ for $0<w_1<w_2<\infty$, it is the two-level $\ell_1$ ``norm" in~\cite{huang2014two}.
\item If $\lambda_i=w_1+w_2(i-1)$, where $w_1\geq0$ and $w_2>0$, it is the small magnitude penalized (SMAP) in~\cite{SMAP2014}.
\end{itemize}

In the following of this section, we introduce three equivalent formulations of the nonconvex sorted $\ell_1$, which will help us to show the convergence results of the methods for solving problems with nonconvex sorted $\ell_1$ terms in the following sections.

We introduce the following three functions with additional variables $\vP$, $\vv$, and $\Lambda$, respectively, and show that the nonconvex sorted $\ell_1$ is equivalent to the minimum of those functions by eliminating the new variables with the corresponding constraints.
\begin{align*}
F_1(\vu,\vP)&=\sum_{i=1}^n(\vP\lambda)_i|u_i|,\\
F_2(\vu,\vv)&=\lambda_1\|\vu\|_1+(\lambda_2-\lambda_1)\|\vu-\vv^1\|_1+(\lambda_3-\lambda_2)\|\vu-\vv^1-\vv^2\|_1\\
&\quad+\cdots + (\lambda_n-\lambda_{n-1})\|\vu-\sum_{j=1}^{n-1}\vv^j\|_1,\\
F_3(\vu,\Lambda)&=\lambda_1\|\vu\|_1+(\lambda_2-\lambda_1)\|\vu\odot\Lambda^1\|_1+(\lambda_3-\lambda_2)\|\vu\odot\Lambda^1\odot\Lambda^2\|_1\\
&\quad +\cdots + (\lambda_n-\lambda_{n-1})\|\vu\odot\Lambda^1\odot\Lambda^2\odot\cdots\odot\Lambda^{n-1}\|_1.
\end{align*}

\begin{theorem}\label{thm:1}Minimizing $F_1(\vu,\vP)$, $F_2(\vu,\vv)$, and $F_3(\vu,\Lambda)$ over $\vP$, $\vv$, and $\Lambda$ with their corresponding constraints given below, respectively, eliminates $\vP$, $\vv$, and $\Lambda$ and obtains the nonconvex sorted $\ell_1$ defined in~\eqref{eqn:nonconvexl1}.
\begin{align*}
R_\lambda(\vu) & = \min_{\vP\in\cP}F_1(\vu,\vP)\\
&= \min_{\{\vv^j\}_{j=1}^{n-1}:\|\vv^j\|_0\leq 1} F_2(\vu,\vv)\\
& = \min_{\{\Lambda^j\}_{j=1}^{n-1}:\Lambda^j_i\in\{0,1\}, \sum_i\Lambda_i^j\geq n-1} F_3(\vu,\Lambda),
\end{align*}
where $\cP$ is the set of all permutation matrices.
\end{theorem}
\begin{proof}
Part 1: Without loss of generality, assume that $|u_1|\geq |u_2|\geq\cdots\geq |u_n|$. We show that $F_1(\vu,\vP)\geq F_1(\vu,\vI_n)$ for all $\vP\in\cP$, which means that $\min\limits_\vP F_1(\vu,\vP)=F_1(\vu,\vI_n)=R_\lambda(\vu)$.

Given $\vP\in\cP$. If $(\vP\lambda)_1\neq \lambda_1=(\vP\lambda)_k$ for some $k> 1$, we can exchange the $1$st and the $k$th components of $\vP\lambda$ and obtain $\vP^1$ such that $(\vP^1\lambda)_1=\lambda_1$, otherwise let $\vP^1=\vP$. We have
\begin{align*}
F_1(\vu,\vP^1)- F_1(\vu,\vP)&=\lambda_1|u_1|+(\vP\lambda)_1|u_{k}|-(\vP\lambda)_{1}|u_1|-\lambda_{1}|u_{k}|\\
&=[\lambda_{1}-(\vP\lambda)_1](|u_1|-|u_{k}|)\leq0.
\end{align*}

Similarly, for $j>1$, we can find $\vP^j$ such that $(\vP^j\lambda)_i=\lambda_i$ for $i=1,\cdots,j$ and $F_1(\vu,\vP^j)\leq F_1(\vu,\vP^{j-1})\leq\cdots\leq F_1(\vu,\vP)$. When $j=n$, we have $\vP^n\lambda=\vI_n\lambda$ and $F_1(\vu,\vI_n)\leq F_1(\vu,\vP).$

Part 2: We show the equivalence by two steps: 1) $F_2(\vu,\vv)\geq R_\lambda(\vu)$ for all $\vv$'s satisfying $\|\vv^j\|_0\leq 1$ and 2) $F_2(\vu,\bar{\vv})= R_\lambda(\vu)$ for some $\bar{\vv}$ satisfying the constraint.

The constraint that $\|\vv^j\|_0\leq1$ for all $j$ gives us
\begin{align*}
\left\|\vu-\sum_{j=1}^{k}\vv^j\right\|_1  \geq  \sum_{i=k+1}^n|\vu|_{[i]},
\end{align*}
for $1\leq k\leq n-1$. Thus for any $\vv$ satisfying the constraint $\|\vv^j\|_0\leq 1$, we have
\begin{align*}
F_2(\vu,\vv)& \geq \lambda_1\sum_{i=1}^n|\vu|_{[i]}+(\lambda_2-\lambda_1)\sum_{i=2}^n|\vu|_{[i]}+(\lambda_3-\lambda_2)\sum_{i=3}^n|\vu|_{[i]}\\
&\quad+\cdots + (\lambda_n-\lambda_{n-1})|\vu|_{[n]}\\
& =\lambda_1|\vu|_{[1]}+\lambda_2|\vu|_{[2]}+\cdots+\lambda_n|\vu|_{[n]} = R_\lambda(\vu).
\end{align*}

Let $\bar\vv^j$ be the vector with all zeros except the component corresponding to the $j$th largest absolute value of $\vu$ having the same value as $\vu$. We have
\begin{align*}
\left\|\vu-\sum_{j=1}^{k}\vv^j\right\|_1  =  \sum_{i=k+1}^n|\vu|_{[i]},
\end{align*}
for $1\leq k\leq n-1$, which gives us $F_2(\vu,\bar\vv)=R_\lambda(\vu)$.

Part 3: The proof is similar to that in Part 2 and we omit it here. \qed
\end{proof}
\begin{remark} The constraint $\Lambda^j_i\in\{0,1\}$ in $F_3(\vu,\Lambda)$ can be relaxed to $\Lambda^j_i\in[0,1]$ without changing the equivalence result. We will use the exact relaxed constraint $\Lambda^j_i\in[0,1]$ for $\Lambda$.
\end{remark}

Define $\cV=\{\vv:\|\vv^j\|_0\leq1\}$ and $\cL=\{\Lambda:\Lambda_i^j\in[0,1], \sum_i\Lambda_i^j\geq n-1\}$. Note that the set $\cP$ only has finite number of points in $\vR^{n\times n}$, thus is not continuous; The set $\cV$ is continuous but non-convex; The set  $\cL$ is continuous and convex.

\begin{remark}\label{remark:13} For each $\vP\in\cP$, we can construct a $\Lambda\in \cL$ such that $F_1(\vu,\vP)=F_3(\vu,\Lambda)$. Each column of $\vJ_n-\vP$ corresponds to one vector $\Lambda^j$, i.e., $\Lambda^j$ is the $j$th column of $\vJ_n-\vP$. In addition, if $\vP\in\cP$ is optimal for $F_1(\vu,\vP)$, the corresponding $\Lambda$ is also optimal for $F_3(\vu,\Lambda)$. Thus, we can consider $F_3(\vu,\Lambda)$ as a relaxation to $F_1(\vu,\vP)$.
\end{remark}

\begin{remark} For each $\vP\in\cP$, we can also construct a $\vv\in\cV$ such that $F_1(\vu,\vp)=F_2(\vu,\vv)$. Choose the support of each column of $\vP$ as the support for the corresponding $\vv^j$, i.e, the supports of the $j$th column of $\vP$ and $\vv^j$ are the same. Then $F_2(\vu,\vv)=F_1(\vu,\vP)$ when $\vv^j$ is the projection of $\vu$ onto the support of $\vv^j$. In addition, if $\vP$ is optimal for $F_1(\vu,\vP)$, then $\vv$ is also optimal for $F_2(\vu,\vv)$. Thus $F_2(\vu,\vv)$ is also a relaxation to $F_1(\vu,\vP)$. It can be shown that if $\vv$ is optimal for $F_2(\vu,\vv)$, we can find an optimal $\vP$ for $F_1(\vu,\vP)$ from the support of $\{\vv^j\}_{j=1}^{n-1}$.
\end{remark}

If the $\ell_1$ terms in~\eqref{sec:L1-minimization} and~\eqref{sec:BPDN} are replaced by the nonconvex sorted $\ell_1$ terms, the basis pursuit problem~\eqref{sec:L1-minimization} becomes
\begin{align*}
\Min_\vu \ R_\lambda(\vu)+\iota_{\{\vu:\vA\vu=\vb\}}(\vu),
\end{align*}
and the unconstrained basis pursuit denoising problem~\eqref{sec:BPDN} becomes
\begin{align}\label{pro:bpdn}
\Min_\vu \ R_\lambda(\vu)+\alpha \|\vA\vu-\vb\|^2.
\end{align}
Combining these two problems together into a more general problem
\begin{align}\label{for:problem}
\Min_\vu\ E(\vu):=\ R_\lambda(\vu)+ L(\vu),
\end{align}
where $L(\vu)$ can be any convex loss function. Except compressive sensing, this general problem has more applications in geophysics, image processing, sensor networks, and computer vision. The interested reader is referred to~\cite{bruckstein2009sparse,wright2010sparse} for a comprehensive review of these applications.

Two lemmas stating sufficient and necessary conditions for $\vu^*$ being a local minimizer of $E(\vu)$ are introduced.

\begin{lemma}\label{lemma1}If $\vu^*$ is  a local minimizer of $E(\vu)$, then for any $\vP^*$ minimizing $F_1(\vu^*,\vP)$, $(\vu^*,\Lambda^*)$ with $\Lambda^*$ being constructed from $\vP^*$ as in Remark~\ref{remark:13} is a local minimizer of $E_3(\vu,\Lambda)$.
\end{lemma}
\begin{proof}Since $\vu^*$ is a local minimizer of $E(\vu)$, we can find $\epsilon>0$ such that for all $\vu$ satisfying $\|\vu-\vu^*\|<\epsilon$, we have $E(\vu)\geq E(\vu^*)$. Therefore, for all $(\vu,\Lambda)$, satisfying $\|(\vu,\Lambda)-(\vu^*,\Lambda^*)\|<\epsilon$, we have $\|\vu-\vu^*\|<\epsilon$. Thus
\begin{align*}
E_3(\vu,\Lambda)\geq E(\vu)\geq E(\vu^*) = E_3(\vu^*,\Lambda^*),
\end{align*}
for all $(\vu,\Lambda)$ such that $\|(\vu,\Lambda)-(\vu^*,\Lambda^*)\|<\epsilon$. Then $(\vu^*,\Lambda^*)$ is a local minimizer of $E_3(\vu,\Lambda)$.\qed
\end{proof}

\begin{lemma}\label{lemma2}Given fixed $\vu^*$, if for all $\bar{\vP}\in \cP$ minimizing $E_1(\vu^*,\vP)$, we also have $\vu^*$ minimizing $E_1(\vu,\bar\vP)$, then $\vu^*$ is a local minimizer of $E(\vu)$.
\end{lemma}
\begin{proof}
There exists $\epsilon>0$ such that when $\|\vu-\vu^*\|<\epsilon$, we can always find the same $\vP\in \cP$ such that both $R_\lambda(\vu)=F_1(\vu,\vP)$ and $R_\lambda(\vu^*)=F_1(\vu^*,\vP)$ are satisfied. Therefore, we have $E(\vu^*)=R_\lambda(\vu^*)+L(\vu^*)=F_1(\vu^*,\bar{\vP})+L(\vu^*)\leq F_1(\vu,\bar{\vP})+L(\vu)=R_\lambda(\vu)+L(\vu)=E(\vu)$.
\qed\end{proof}
With these two lemmas, we propose two algorithms for solving problem~\eqref{for:problem} with nonconvex sorted $\ell_1$ term in the next two sections.

%%% Iteratively Reweighted $\ell_1$ Minimization %%%
\section{Iteratively Reweighted $\ell_1$ Minimization}\label{sec:alg1}

%\subsection{Algorithm}

It is difficult to solve problem~\eqref{for:problem} directly because of the non-convexity of $R_\lambda(\vu)$. In this section, we apply the equivalence results in the previous section to solve problem~\eqref{for:problem}. Problem~\eqref{for:problem} is equivalent to the following two problems:
\begin{align}
&\Min_{\vu,\vP\in \cP}\ E_1(\vu,\vP):= F_1(\vu,\vP)+L(\vu),\label{for:problem1}\\
&\Min_{\vu,\Lambda\in \cL}\ E_3(\vu,\Lambda):= F_3(\vu,\Lambda)+L(\vu).\label{for:problem3}
\end{align}
There are two variables in $E_1(\vu,\vP)$ (or $E_3(\vu,\Lambda)$) and we can apply the alternating minimization procedure to solve this problem because the problem is easy to solve with one of the two variables is fixed. Fix $\vu$, the variable $\vP$ (or $\Lambda$) can be obtained in closed-form; while the problem for $\vu$ with $\vP$ (or $\Lambda$) fixed can be formulated into a weighted $\ell_1$ minimization problem which is convex and for which there are many existing solvers. The step of updating $\vP$ (or $\Lambda$) is just updating the weights in the weighted $\ell_1$ minimization, and this procedure is essentially an iteratively reweighted $\ell_1$ minimization.

The set of all the weights is fixed and each reweighting is just a permutation of the weights. This is different from the iteratively reweighted $\ell_1$ minimization in~\cite{candes2008enhancing,chartrand2014shrinkage} where each weight is updated independently from the corresponding component of the previous result, i.e., no sorting is needed and the set of the weights is not fixed. The proposed iteratively reweighted $\ell_1$ minimization for~\eqref{for:problem} is summarized in algorithm~\ref{alg:IRL1}, whose convergence result is shown below.
\begin{algorithm}[H]
\caption{Iteratively Reweighted $\ell_1$ Minimization}\label{alg:IRL1}
\label{alg:ADM-1}
\begin{algorithmic}
\State Initialize $\lambda$, $\vu^0$
\For {$l =0,1, \cdots$}
\State Update $\vP^l=\argmin\limits_\vP F_1(\vu^l,\vP)$ with an optimal $\vP$ such that $\vP^l$ is different from $\{\vP^0$, $\vP^1, \cdots, \vP^{l-1}\}$. If there is no optimal $\vP$ satisfying this condition, break.
\State Update $\vu^{l+1}=\argmin\limits_\vu E_1(\vu,\vP^l)$.
\EndFor
\end{algorithmic}
\end{algorithm}
\begin{remark}Because $F_1(\vu,\vP)=\sum_i(\vP\lambda)_i|u_i|$, which depends on $\vP\lambda$, not $\vP$. If there are equivalent components in $\lambda$, then different $\vP$'s may have the same $\vP\lambda$. In the algorithm, we just choose the optimal $\vP^l$ such that $\vP^l\lambda$ is different from \{$\vP^0\lambda$, $\vP^1\lambda, \cdots, \vP^{l-1}\lambda$\}.
\end{remark}

\begin{remark}
The initial $\vu^0$ can be chosen as the output of the $\ell_1$ minimization problem: $\Min\limits_\vu\|\vu\|_1+L(\vu)$. When we use alternating minimization procedure on $F_3(\vu,\Lambda)$, in order to minimize $F_3(\vu^l,\Lambda)$, we can first find an optimal $\vP^l$ from minimizing $F_1(\vu^l,\vP)$ and then construct the corresponding optimal $\Lambda^l$ using Remark~\ref{remark:13}.
\end{remark}

We show that the proposed iteratively reweighted $\ell_1$ minimization can obtain a local optimum of problem~\eqref{for:problem} in finite steps.
\begin{theorem}\label{thm:con1}
Algorithm~\ref{alg:IRL1} will converge in finite steps, and the output $\vu^*$ is a local minimizer of $E(\vu)$. In addition, $(\vu^*,\Lambda^*)$ is a local minimizer of $E_3(\vu,\Lambda)$.
\end{theorem}
\begin{proof}Since $\vP\in\cP$, there are only finite number of $\vP$'s (the total number of different $\vP$'s is $n!$), and the algorithm will stop in finite steps. Assume that the algorithm stops at step $\bar l$. There are two cases:
\begin{itemize}
\item Case I: $E_1(\vu^{{\bar l}-1},\vP^{{\bar l}-2})>E_1(\vu^{\bar l},\vP^{{\bar l}-1})$. In this case, there is no element in $\{\vP^0,\vP^1,\cdots,\vP^{{\bar l}-2}\}$ being a minimizer of $F_1(\vu^{\bar l},\vP)$ because if $\vP^l$ is a minimizer of $F_1(\vu^{\bar l},\vP)$ for some integer $l\in[0, {\bar l}-2]$, we have
\begin{align*}
E(\vu^{\bar l})&=E_1(\vu^{\bar l},\vP^l)\geq E_1(\vu^{l+1},\vP^l)\\
&\geq E_1(\vu^{{\bar l}-1},\vP^{{\bar l}-2})>E_1(\vu^{\bar l},\vP^{{\bar l}-1})\geq E(\vu^{\bar l}).
\end{align*}
In addition, $\vP^{{\bar l}-1}$ is a minimizer of $F_1(\vu^l,\vP)$, otherwise, we can find an optimal $\vP^{\bar l}$ such that $E_1(\vu^{\bar l},\vP^{\bar l})=E(\vu^{\bar l})< E_1(\vu^{\bar l},\vP^{{\bar l}-1})$. Therefore, $\vP^{{\bar l}-1}$ is the unique minimizer of $F_1(\vu^l,\vP)$, $\vu^{\bar l}$ being a local minimizer of $E(\vu)$ comes from Lemma~\ref{lemma2}.
\item Case II: $E_1(\vu^{l_0-1},\vP^{l_0-2})>E_1(\vu^{l_0},\vP^{l_0-1})=E_1(\vu^{l_0+1},\vP^{l_0})=\cdots=E_1(\vu^{\bar l},\vP^{{\bar l}-1})$ for $l_0\leq {\bar l}-1$. In the same way, we can show that there is no element in $\{\vP^0,\vP^1,\cdots,\vP^{l_0-2}\}$ being a minimizer of $F_1(\vu^{\bar l},\vP)$. In addition, $\vP^{{\bar l}-1}$ is a minimizer of $F_1(\vu^{\bar l},\vP)$. For $l\in [l_0-1, {\bar l}-2]$, we have
\begin{align*}
E(\vu^{\bar l}) = E_1(\vu^{\bar l},\vP^{{\bar l}-1})=E_1(\vu^{l+1},\vP^{l})\leq E_1(\vu^{{\bar l}},\vP^{l}).
\end{align*}
For any $l\in[l_0-1,{\bar l}-2]$ such that the equality satisfies, $\vP^{l}$ is a minimizer of $F_1(\vu^{\bar l},\vP)$, and $\vu^{\bar l}$ is also a minimizer of $E_1(\vu,\vP^l)$. In addition, there is no more minimizer of $F_1(\vu^{\bar l},\vP)$. Then Lemma~\ref{lemma2} tells us that $\vu^{\bar l}$ is a local minimizer.
\end{itemize}
If $\vu^*$ is a local minimizer, then $(\vu^*,\Lambda^*)$ being a local minimizer of $F_3(\vu,\Lambda)$ follows from Lemma~\ref{lemma1}.
\qed\end{proof}

Notice that the function value may not be strictly decreasing, and the next theorem discusses the case when the function values of two consecutive iterations are the same.

\begin{theorem}If $E_1(\vu^{l-1},\vP^{l-2})=E_1(\vu^{l},\vP^{l-1})$ for some $l\leq {\bar l}$, $(\vu^{l-1},\vP^{l-2})$ is a coordinatewise minimum point of $E_1(\vu,\vP)$, i.e., $\vu^{l-1}$ is a minimizer of $E_1(\vu,\vP^{l-2})$ and $\vP^{l-2}$ is a minimizer of $E_1(\vu^{l-1},\vP)$.
\end{theorem}

\begin{proof}
$E_1(\vu^{l-1},\vP^{l-2})=E_1(\vu^{l},\vP^{l-1})$, together with the nonincreasing property of the algorithm
\begin{align*}E_1(\vu^{l-1},\vP^{l-2})\geq E_1(\vu^{l-1},\vP^{l-1})\geq E_1(\vu^{l},\vP^{l-1}),\end{align*}
gives us
\begin{align*}E_1(\vu^{l-1},\vP^{l-2})= E_1(\vu^{l-1},\vP^{l-1})= E_1(\vu^{l},\vP^{l-1}).\end{align*}
Thus
\begin{align*}E_1(\vu^{l-1},\vP^{l-2})&=E_1(\vu^{l-1},\vP^{l-1})=\min_\vP E_1(\vu^{l-1},\vP).
\end{align*}
In addition, the algorithm gives
\begin{align*}E_1(\vu^{l-1},\vP^{l-2})&=\min_\vu E_1(\vu,\vP^{l-2}).\end{align*}
Thus $(\vu^{l-1},\vP^{l-2})$ is a coordinatewise minimum point of $E_1(\vu,\vP)$.
\qed\end{proof}

\begin{remark}If there exists $l<{\bar l}$ such that $E_1(\vu^{l-1},\vP^{l-2})=E_1(\vu^{l},\vP^{l-1})>E_1(\vu^{{\bar l}},\vP^{l-1})$, then $\vu^{l-1}$ and $\vu^l$ may not be local minimizers of $E(\vu)$.
\end{remark}
%\commmy{Perfect if an example is given for this remark.}

%%% Iterative Sorted Thresholding %%%
\section{Iterative Sorted Thresholding} \label{sec:alg2}

In this section, we propose another iterative method for solving problem~\eqref{for:problem}. There are several iterative thresholding algorithms for compressive sensing in literature: iterative shrinkage-thresholding for $\ell_1$~\cite{daubechies2004iterative}, iterative hard thresholding~\cite{blumensath2009iterative}, iterative half thresholding for $\ell_{1/2}$~\cite{xu2012regularization}. In this section, we shall establish a thresholding algorithm for problem~\eqref{for:problem}, which generalizes iterative shrinkage-thresholding and iterative hard thresholding.

Assume that $L:\vR^n\rightarrow \vR$ is a smooth, bounded below, convex function of type $C^{1,1}$, i.e., continuously differentiable with Lipschitz continuous gradient:
\begin{align*}
\|\nabla L(\vu)-\nabla L(\vv)\|\leq L_L	\|\vu-\vv\|
\end{align*}
for all $\vu,\vv\in\vR^n$. The iterative sorted thresholding is described in Algorithm~\ref{alg:IST}, and its convergence is shown in Theorem~\ref{thm:ist_converge}.
\begin{algorithm}[H]
\caption{Iteratively Sorted Thresholding}\label{alg:IST}
\begin{algorithmic}
\State Initialize $\vu^0$
\For {$l =0,1, \cdots$}
\State Find $\vu^{l+1}=\argmin_\vu\, \beta R_\lambda(\vu)+{1\over 2}\|\vu-(\vu^l-\beta\nabla L(\vu^l))\|^2$.
\EndFor
\end{algorithmic}
\end{algorithm}

\begin{remark}The iteration can be expressed as
\begin{align*}
\vu^{l+1}\in \prox_{\beta R_\lambda}(\vu^l-\beta\nabla L(\vu^l))
\end{align*}
and the proximal operator can be evaluated in closed form (see Lemma~\ref{lemma:prox}).
\end{remark}

\begin{lemma}\label{lemma:prox}The proximal operator of $R_\lambda$ can be evaluated as
\begin{align*}
\prox_{R_\lambda}(\vx)=\max(|\vx|-\vP\lambda,\vzero)\odot\sign(\vx),
\end{align*}
for any $\vP\in\cP$ such that $R_\lambda (\vx)=\sum_{i=1}^n(\vP\lambda)_ix_i$. Here $\max$ and $\sign$ are both component-wise.
\end{lemma}

\begin{proof}Evaluating the proximal operator is equivalent to solving the problem
\begin{align}\label{for:ist_u}
\Min_\vu\, R_\lambda(\vu)+{1\over 2}\|\vu-\vx\|_2^2,
\end{align}
which is the same as the following problem with additional variable $\vP$:
\begin{align}\label{for:ist_up}
\Min_{\vu,\vP}\, \sum_{i=1}^n(\vP\lambda)_i|u_i|+{1\over 2}\|\vu-\vx\|_2^2.
\end{align}
To solve problem~\eqref{for:ist_up}, we eliminate the variable $\vu$ from the objective function.
\begin{align*}
&\,\quad\min_\vu\sum_{i=1}^n(\vP\lambda)_i|u_i|+{1\over 2}\|\vu-\vx\|_2^2\\
&=\min_\vu\sum_{i=1}^n\lambda_i|(\vP^T\vu)_i|+{1\over 2}\sum_{i=1}^n((\vP^T\vu)_i-(\vP^T\vx)_i)^2\\
&=\sum_{i=1}^nf_{\lambda_i}((\vP^T\vx)_i),
\end{align*}
where
\begin{align*}
f_{\lambda_i}(x)=\left\{\begin{array}{ll}{1\over 2}x^2, &\mbox{ if } |x|< \lambda_i, \\ \lambda_i|x|-{1\over 2}\lambda_i^2, &\mbox{ if }|x|\geq \lambda_i.\end{array}\right.
\end{align*}
The optimal $\vu$ is chosen to be
\begin{align*}
\vu&=\vP(\vP^T\vu)=\vP[\max(|\vP^T\vx|-\lambda,\vzero)\odot\sign(\vP^T\vx)]\\
&=\max(|\vx|-\vP\lambda,\vzero)\odot\sign(\vx).
\end{align*}

Next, we shall find an optimal $\vP$ for problem~\eqref{for:ist_up}. If $\lambda_i<\lambda_j$, the function $f_{\lambda_i}-f_{\lambda_j}$ can be expressed as
\begin{align*}
f_{\lambda_i}(x)-f_{\lambda_j}(x)=\left\{\begin{array}{ll}0, &\mbox{ if } |x|< \lambda_i, \\-{1\over2}(|x|-\lambda_i)^2, & \mbox{ if } \lambda_i\leq|x|\leq \lambda_j,\\ -{1\over2}(\lambda_j-\lambda_i)(2|x|-\lambda_i-\lambda_j), &\mbox{ if }|x|> \lambda_j.\end{array}\right.
\end{align*}
Therefore, the even function $f_{\lambda_i}(x)-f_{\lambda_j}(x)$ is decreasing when $x\geq 0$, i.e.,
\begin{align}\label{for:ist_1}
f_{\lambda_i}(x_i)-f_{\lambda_j}(x_i) \geq f_{\lambda_i}(x_j) -f_{\lambda_j}(x_j)
\end{align}
if $|x_i|\leq |x_j|$. Rearranging~\eqref{for:ist_1} gives us
\begin{align}
f_{\lambda_i}(x_i)+f_{\lambda_j}(x_j) \geq f_{\lambda_i}(x_j)+ f_{\lambda_j}(x_i).
\end{align}
Applying the same technique in part I of the proof for Theorem~\ref{thm:1}, we show that $|\vP^T\vx|$ is in decreasing order for optimal $\vP$'s. The definition of $R_\lambda$ gives us that $R_\lambda(\vx)=\sum_{i=1}^n\lambda_i|(\vP^T\vx)_i|=\sum_{i=1}^n(\vP\lambda)_i|x_i|$ for optimal $\vP$'s.
\qed\end{proof}

\begin{remark}The proximal operator can be multi-valued. For different optimal $\vP$, the output can be different. For example, when $\lambda = (0,1)$ and $\vx=(1,1)$, then both $(1,0)$ and $(0,1)$ are optimal for problem~\eqref{for:ist_u}.
\end{remark}

Before proving the convergence, we state and prove two lemmas.

\begin{lemma}\label{lemma4}If there exists $\vu^*$ such that
\begin{align}\label{for:ist_lemma4_1}
\vu^*\in\prox_{R_\lambda}(\vw^*),
\end{align}
then
\begin{itemize}
\item $|w^*_i|\geq |w^*_j|$ if $|u^*_i|>|u^*_j|$;
\item $|w^*_i|= |w^*_j|$ if $|u^*_i|=|u^*_j|>0$;
\item In addition, if $|\vu^*|$ is in decreasing order, i.e., $|u^*_1|\geq |u^*_2|\geq \cdots\geq|u^*_{I-1}|>|u^*_I|=\cdots=|u^*_n|=0$, we have $|w^*_1|\geq |w^*_2|\geq \cdots\geq|w^*_{I}|$, and $|w^*_i|\geq|w^*_j|$ for all $i<I$ and $j\geq I$.
\end{itemize}
\end{lemma}

\begin{proof} Assumption~\eqref{for:ist_lemma4_1} ensures the existence of $\vP^*\in \{\vP\in\cP:R_\lambda(\vw^*)=\sum_{k=1}^n(\vP\lambda)_k|w^*_k|\}$ such that
\begin{align}\label{for:ist_lemma4_2}
|\vu^*|=\max(|\vw^*|-\vP^*\lambda,\vzero).
\end{align}

Part I:
If $|u^*_i|>|u^*_{j}|$,~\eqref{for:ist_lemma4_2} shows that
\begin{align}\label{for:ist_lemma4_3}
|w^*_i|-(\vP^*\lambda)_i =|u^*_i|>|u^*_j|\geq|w^*_j|-(\vP^*\lambda)_j.
\end{align}
Proof by contradiction: Assume $|w^*_i|< |w^*_j|$, then $(\vP^*\lambda)_i\geq(\vP^*\lambda)_j$. Therefore,
\begin{align*}
|w^*_i|-(\vP^*\lambda)_i <|w^*_j|-(\vP^*\lambda)_j,
\end{align*}
which is a contradiction to~\eqref{for:ist_lemma4_3}. Thus $|w^*_i|\geq|w^*_j|$.

Part II: If $|u^*_i|=|u^*_j|>0$,~\eqref{for:ist_lemma4_2} shows that
\begin{align}\label{for:ist_lemma4_4}
|w^*_i|-(\vP^*\lambda)_i =|u^*_i|=|u^*_j|= |w^*_j|-(\vP^*\lambda)_j.
\end{align}
Proof by contradiction: Without loss of generality, assume $|w^*_i|<|w^*_j|$, we have $(\vP^*\lambda)_i\geq (\vP^*\lambda)_j$. Therefore,
\begin{align*}
|w^*_i|-(\vP^*\lambda)_i < |w^*_j|-(\vP^*\lambda)_j,
\end{align*}
which is a contradiction to~\eqref{for:ist_lemma4_4}. Thus $|w^*_i|=|w^*_j|$.

Part III: Proof by contradiction: Assume $|w^*_i|<|w^*_{i+1}|$ for $i<I$, we have $(\vP^*\lambda)_i\geq(\vP^*\lambda)_{i+1}$. Then
\begin{align}
|u^*_i|=|w^*_i|-(\vP^*\lambda)_i<|w^*_{i+1}|-(\vP^*\lambda)_{i+1}\leq |u^*_{i+1}|,
\end{align}
which a contradiction to the assumption that $|\vu^*|$ is in decreasing order. Thus $|w^*_i|\geq |w^*_{i+1}|$ for all $i<I$.

In addition if $i<I$ and $j\geq I$, then $|u^*_i|>|u^*_j|=0$, Part I gives us that $|w^*_{i}|\geq |w^*_j|$.
\qed\end{proof}

\begin{lemma}\label{lemma5} If there exists $\vu^*$ such that
\begin{align}\label{for:ist_iff_1}
\vu^*\in\prox_{\beta R_\lambda}(\vu^*-\beta\nabla L(\vu^*)),
\end{align}
then $\vu^*$ is a local minimizer of $E(\vu)$.
\end{lemma}

\begin{proof}Without loss of generality, we assume that $|\vu^*|$ is in decreasing order. If $|\vu^*|$ is not in decreasing order, we can find $\vP\in \cP$ such that $|\vP^T\vu^*|$ is in decreasing order and reformulate function $L$ accordingly. When $L(\vu)=\alpha\|\vA\vu-\vb\|^2$, we can just rearrange the columns of $\vA$. %, and~\eqref{for:ist_iff_1} is equvalent to

Let $\vw^*=\vu^*-\beta\nabla L(\vu^*)$ and define two subsets of $\cP$ as follows:
\begin{align*}
\cP^*:&=\{\vP\in\cP:R_\lambda(\vw^*)=\sum_{i=1}^n(\vP\lambda)_i|w^*_i|\},\\
\hat\cP:&=\{\vP\in\cP:R_\lambda(\vu^*)=\sum_{i=1}^n(\vP\lambda)_i|u^*_i|\}.
\end{align*}
Assumption~\eqref{for:ist_iff_1} ensures the existence of $\vP^*\in\cP^*$ such that
\begin{align}\label{for:assume}
\vu^*=\max(|\vw^*|-\beta\vP^*\lambda,\vzero)\odot\sign(\vw^*).
\end{align}
First of all, we show that $\vP^*\in\hat\cP$, i.e., $(\vP^*\lambda)_i\leq (\vP^*\lambda)_j$ if $|u^*_i|>|u^*_j|$. If $|u^*_i|>|u^*_{j}|$,~\eqref{for:assume} shows that
\begin{align}\label{for:ist_lm_1}
|w^*_i|-\beta(\vP^*\lambda)_i =|u^*_i|>|u^*_j|\geq|w^*_j|-\beta(\vP^*\lambda)_j.
\end{align}
Proof by contradiction: Assume $(\vP^*\lambda)_i> (\vP^*\lambda)_j$, then $|w^*_i|\leq|w^*_j|$ because $\vP^*\in\cP^*$. Therefore,
\begin{align*}
|w^*_i|-\beta(\vP^*\lambda)_i <|w^*_j|-\beta(\vP^*\lambda)_j,
\end{align*}
which is a contradiction to~\eqref{for:ist_lm_1}. Thus $\vP^*\in\hat\cP$.

In order to show that $\vu^*$ is a local minimizer, we shall show that
\begin{align}\label{for:ist_lemma5_2}
\vu^*=\max(|\vw^*|-\beta\vP\lambda,\vzero)\odot\sign(\vw^*)
\end{align}
for all $\vP\in\hat\cP$ from Lemma~\ref{lemma2}.

Comparing~\eqref{for:assume} and~\eqref{for:ist_lemma5_2}, in order to prove the second one, we need to show
\begin{itemize}
\item $(\vP\lambda)_i=(\vP^*\lambda)_i$ for $i<I$ and $\vP\in\hat\cP$,
\item $|w^*_j|-\beta(\vP\lambda)_j\leq 0$ for $j\geq I$ and $\vP\in\hat\cP$,
\end{itemize}
where $I=\min\{i:|u_i|=0\}$.

Firstly, we consider the case when $i<I$. Combining the results in Lemma~\ref{lemma4} and $\vP^*\in\hat\cP$, we have
\begin{align*}
&|u^*_1|\geq |u^*_2|\geq\cdots\geq |u^*_{I-1}|> |u^*_j|=0,\\
&|w^*_1|\geq |w^*_2|\geq\cdots\geq |w^*_{I-1}|\geq |w^*_j|,\\
&(\vP^*\lambda)_1\leq(\vP^*\lambda)_2\leq \cdots\leq(\vP^*\lambda)_{I-1}\leq (\vP^*\lambda)_{j},
\end{align*}
for $j\geq I$. For any $\vP\in\hat\cP$ and $i<I\leq j$, $|u^*_i|>|u^*_j|=0$ implies $(\vP\lambda)_i\leq (\vP\lambda)_j$. Therefore, the set $\{(\vP\lambda)_i\}_{i=1}^{I-1}$ is the same for all $\vP\in\hat\cP$. We shall show that $(\vP\lambda)_{i-1}\leq(\vP\lambda)_{i}$ for all $i<I$ and $\vP\in\hat\cP$.

Proof by contradiction: Assume that $i_0$ is the first index such that $(\vP\lambda)_{i_0}>(\vP\lambda)_{i_0+1}$ for some $\vP\in\hat\cP$. Then there exists $i_1<I$ such that $(\vP\lambda)_{i_0}>(\vP\lambda)_{i_1}$ and $|u^*_{i_0}|>|u^*_{i_1}|$, which contradicts $\vP\in\hat\cP$.

Secondly, we consider the case when $j\geq I$. $\vP^*\in\cP^*$ implies $\max_{j\geq I}|w^*_j|\leq \min_{j\geq I}(\beta\vP^*\lambda)_j$. Thus for any $\vP\in\hat\cP$, we have $|w^*_j|\leq (\beta\vP\lambda)_j$ for all $j\geq I$.
\qed\end{proof}

\begin{theorem}\label{thm:ist_converge}
If $\beta<1/L_L$, the iterative sorted thresholding converges to a local optimum of~\eqref{for:problem}.
\end{theorem}
\begin{proof} From the iterations, we have
\begin{align}
& E(\vu^{l+1})+\left({1\over2\beta}-{L_L\over 2}\right)\|\vu^{l+1}-\vu^l\|_2^2\nonumber\\
\leq  &R_\lambda(\vu^{l+1}) + L(\vu^{l})+(\vu^{l+1}-\vu^l)^T\nabla L(\vu^l)+{ L_L\over2}\|\vu^{l+1}-\vu^l\|_2^2\nonumber\\
&+\left({1\over2\beta}-{L_L\over 2}\right)\|\vu^{l+1}-\vu^l\|_2^2 \label{for:ist_dec_1} \\
= &R_\lambda(\vu^{l+1}) + L(\vu^{l})+(\vu^{l+1}-\vu^l)^T\nabla L(\vu^l)+{ 1\over2\beta}\|\vu^{l+1}-\vu^l\|_2^2\nonumber\\
=& R_\lambda(\vu^{l+1}) + L(\vu^{l})+{ 1\over2\beta}\|\vu^{l+1}-\vu^l+{\beta}\nabla L(\vu^l)\|_2^2-{\beta\over2}\|\nabla L(\vu^l)\|_2^2\nonumber\\
\leq &R_\lambda(\vu^{l}) + L(\vu^{l})+{ 1\over2\beta}\|\vu^{l}-\vu^l+{\beta}\nabla L(\vu^l)\|_2^2-{\beta\over2}\|\nabla L(\vu^l)\|_2^2\label{for:ist_dec_2}\\
= &R_\lambda(\vu^{l}) + L(\vu^{l})=E(\vu^l).\nonumber
\end{align}
Here, \eqref{for:ist_dec_1} comes from the assumption that $L\in C^{1,1}$, and~\eqref{for:ist_dec_2} holds because $\vu^{l+1}$ is a minimizer of $\beta R_\lambda(\vu)+{1\over2}\|\vu-(\vu^l-\beta \nabla L(\vu^l))\|_2^2$. Therefore $E(\vu^l)$ is decreasing when $\beta< 1/L_L$. In addition, $E(\vu^l)$ is bounded below. Then $E(\vu)$ converges to $E(\bar\vu)$, where $\bar\vu$ is a fixed point, i.e., $\bar\vu\in\prox_{\beta R_\lambda}(\bar\vu-\beta \nabla L(\bar\vu))$. Lemma~\ref{lemma5} states that $\bar\vu$ is a local minimizer of $E(\vu)$. Furthermore, all subsequence limit points of $\vu^l$ are fixed points.
\qed\end{proof}

%%% Numerical Experiment %%%
\section{Numerical Experiment}\label{sec:numerical}

We, in this section, illustrate the performance of the proposed nonconvex sorted $\ell_1$ minimization. All the experiments are done in Matlab R2013a with Core i7-3.40 GHz and 16.0G RAM. As discussed previously, if we only give two different weights, i.e., $\lambda_1 = \cdots = \lambda_K = \omega_1$ and $\lambda_{K+1} = \cdots = \lambda_n = 1$, the nonconvex sorted $\ell_1$ minimization becomes the two-level $\ell_1$ minimization (2level)~\cite{huang2014two}. Furthermore, if one sets $\omega_1 = 0$, it reduces to the iterative support detection (ISD)~\cite{wang2010sparse} with fixed number of supports. The most general case is the multiple-level $\ell_1$ minimization (mlevel) and we in this paper use the following strategy:
\[
\lambda_i =
\left \{
\begin{array}{ll}
1, & i > K,\\
e^{- r(K -i)/K}, & i \leq K,
\end{array}
\right.
\]
where $r$ controls the rate of decreasing $\lambda_i$ from $1$ to $0$. The aim of this experimental section is to test different weight setting methods. The comparative method is the iterative reweighted $\ell_1$ minimization (IRL1) in~\cite{candes2008enhancing}, which is also nonconvex but the weights are set according to value not the sort.

The tuning parameters in nonconvex sorted $\ell_1$ minimization, such as $K$, $w_1$, and $r$, are related to the sparsity and nonconvexity. Generally, with the decreasing of $w_1$ and the increasing of $K$ (as long as $K < n/2$), the nonconvexity of $R_\lambda$ increases. In practice, it is better to start from the $\ell_1$ minimization, of which the result is denoted by $\vu^0$, and increase the nonconvexity during the iterations. For 2level and mlevel, we fix $K$ to be $\lfloor \|\vu^0\|_0/3 \rfloor$, i.e., the largest integer smaller than one third of $\#\{i: u_i^0 \neq 0\}$, but change $\omega_1$ for 2level, and $r$ for mlevel as follows:
\begin{eqnarray*}
&& \omega^0_1 = 0.5 ~~\mathrm{and}~~ \omega^l_1 = \max\{0.1, 0.9\omega^{l-1}_1\}, \forall l \geq 1;\\
&& r^l = \min\{ 10, 0.15l\}, \forall l \geq 1.
\end{eqnarray*}
For ISD, $w_1 = 0$ is fixed and we enhance the sparsity by increasing $K$ following the strategy used in \cite{wang2010sparse} \cite{yang2011alternating}. For IRL1, the weight is set as
\[
\lambda^l_i = \frac{1}{ |x_i| + \max \{{0.5}^{l-1}, 0.5^{8}\}},
\]
which is also related to the iteration count.

We first illustrate the phase transition diagram via nonconvex sorted $\ell_1$ minimization. In this experiment, 100-dimensional $s$-sparse signals $\vu$, i.e., $\vu$ has only $s$ nonzero components, and matrix $\vA$ are randomly generated. The non-zero components of $\vu$ and the elements of $\vA$ come from the standard Gaussian distribution. Then we let $\vb = \vA \vu$ and use Algorithm 1 to recover $\vu$ from $\vb$ and $\vA$. Algorithm 1 involves a series of weighted $\ell_1$ minimization, for which we apply YALL1 \cite{yang2011alternating}. ISD and IRL1 also can be solved by YALL1. Those reweighted methods iteratively minimize the weighted $\ell_1$ penalty and the maximum iteration is controlled by $l_{\mathrm{max}}$. IRL1 takes more time than other methods on the weighted L1 minimization problems because the dynamic range of the weights can be very high for IRL1 when the nonzero components are Gaussian distributed. Hence we set $l_{\mathrm{max}} = 3$ for IRL1 and $l_{\mathrm{max}} = 10$ for others. One will observe that even with this setting, the computational time of IRL1 is more than other methods.

After recovering the signal, we calculate the $\ell_\infty$ distance between $\vu$ and the recovered signal. If the distance is smaller than $10^{-3}$, we claim the signal is successfully recovered. In the phase transition diagram, we test $100$ trials for each $s$ and $m$
and then show the results in Fig.\ref{fig-phase} for ISD, 2level, mlevel, and IRL1. For a clear comparison, we in Fig.\ref{fig-phase:d} plot $50 \%$
successful recovery lines for all these four methods.

\begin{figure}[hbtp]
  \centering
  \subfigure[]{
    \label{fig-phase:a}
    \psfrag{k}[c]{\footnotesize sparsity $s$}
    \psfrag{m}[c][BI][1][0]{\footnotesize number of measurements $m$}
    \includegraphics[width=2.2in]{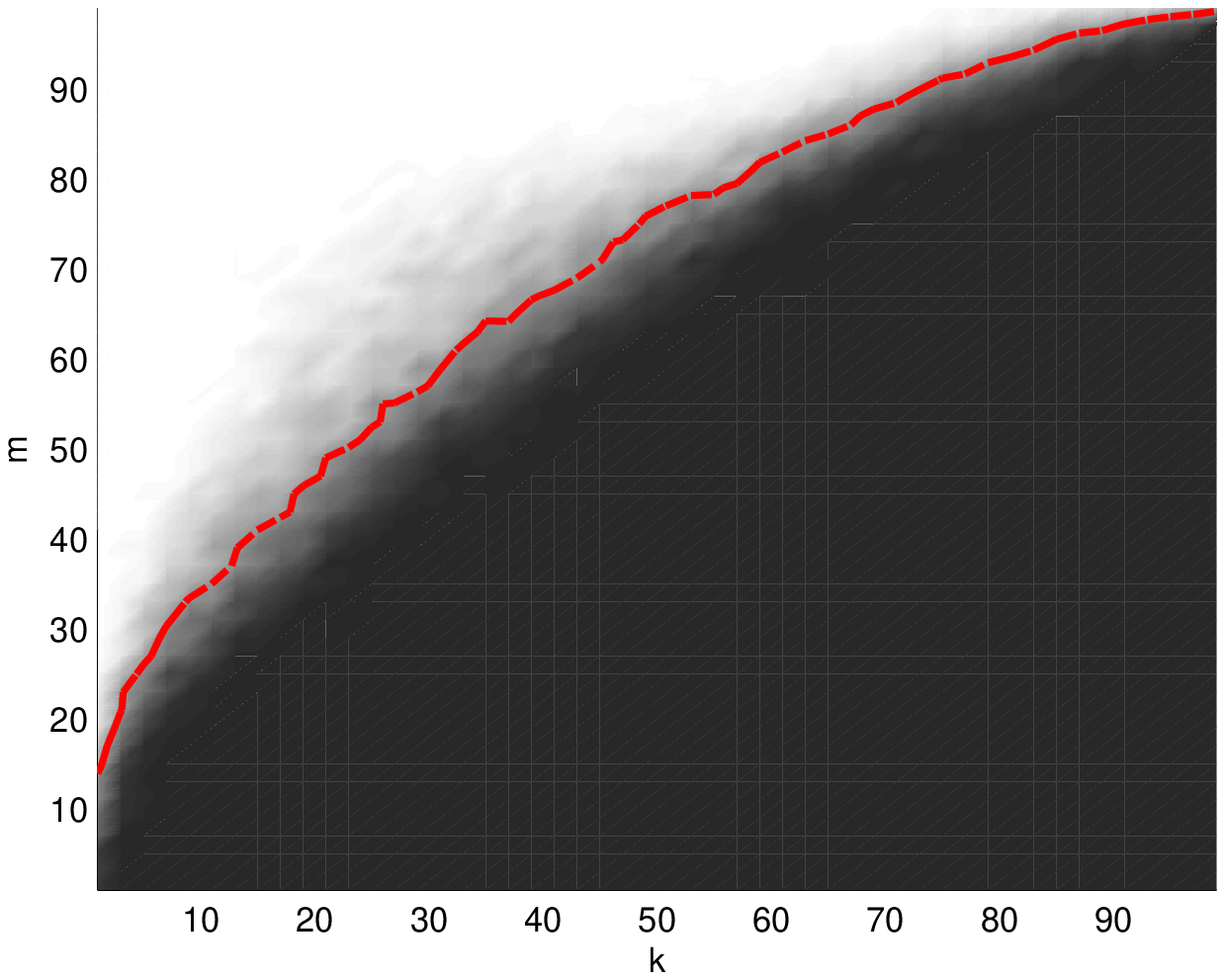}} \quad
  \subfigure[]{
    \label{fig-phase:b}
    \psfrag{k}[c]{\footnotesize sparsity $s$}
    \psfrag{m}[c][BI][1][0]{\footnotesize number of measurements $m$}
    \includegraphics[width=2.2in]{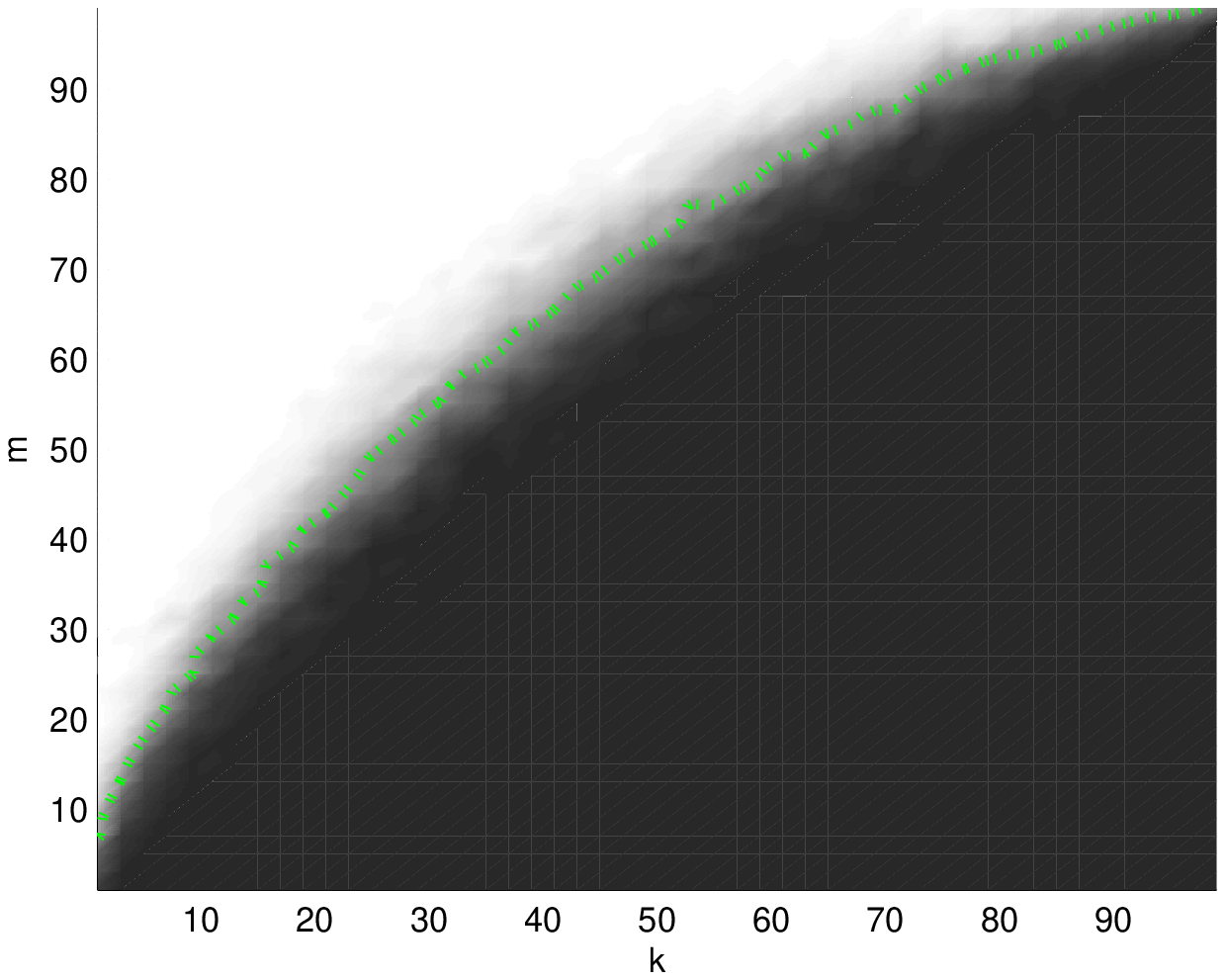}} \\
  \subfigure[]{
    \label{fig-phase:c}
    \psfrag{k}[c]{\footnotesize sparsity $s$}
    \psfrag{m}[c][BI][1][0]{\footnotesize number of measurements $m$}
    \includegraphics[width=2.2in]{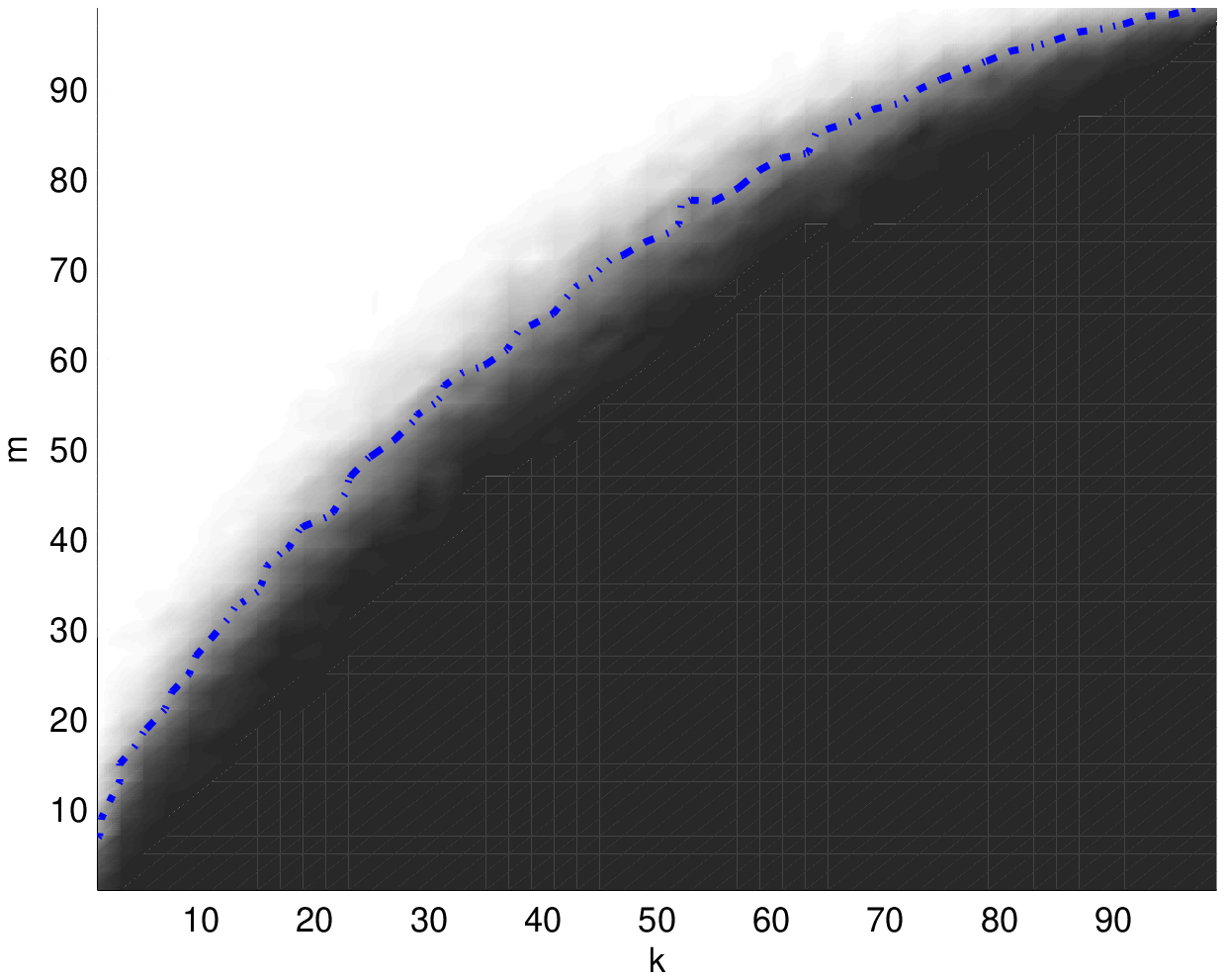}} \quad
  \subfigure[]{
    \label{fig-phase:d}
    \psfrag{k}[c]{\footnotesize sparsity $s$}
    \psfrag{m}[c][BI][1][0]{\footnotesize number of measurements $m$}
    \includegraphics[width=2.2in]{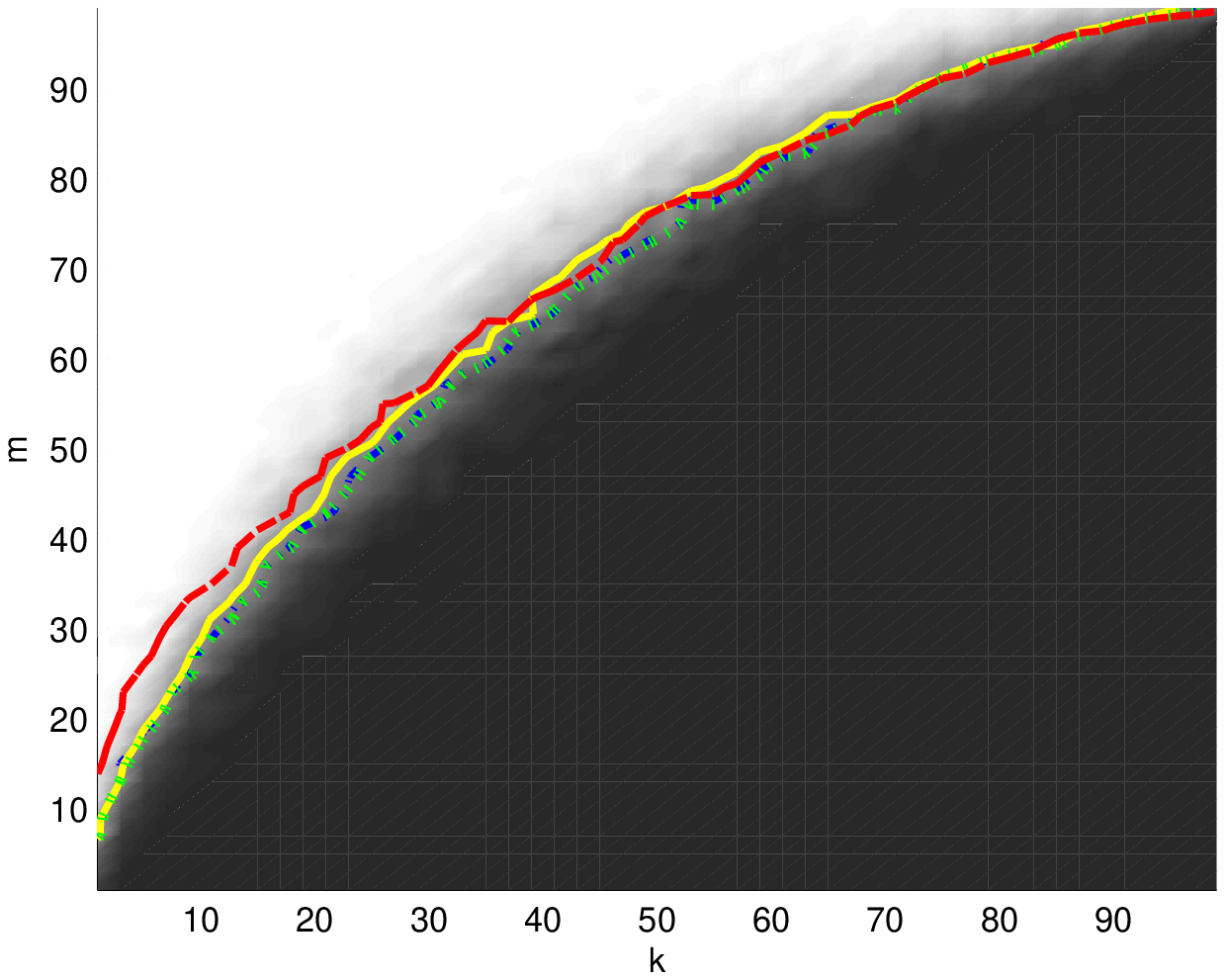}}
    \caption{Phase transition diagram for the considered algorithms: (a) ISD; (b) 2level; (c) mlevel; (d) IRL1. The grey level stands for the successful recovery percentage: white means $100\%$ recovery and black means $0\%$ recovery, the $50\%$ recovery is also displayed by the green solid line. The red, blue, green and yellow lines show the $50\%$ successful recovery by ISD, 2level, mlevel, and IRL1, respectively. In (d), the lines are shown together for comparison.}\label{fig-phase}
\end{figure}

We also repeat the experiment used in \cite{candes2008enhancing}, where a recovery problem with $m = 100$ and $n = 256$ is considered. $\vu$ is a $s$-sparse signal and the recovery performance of different $s$ values is evaluated. In Fig.\ref{fig-recovery:a} and Fig.\ref{fig-recovery:b}, the recovery percentage of different methods and the computational time are shown, respectively. From both Fig.\ref{fig-phase} and Fig.\ref{fig-recovery}, one can find that compared with setting weights by value, setting weights according to the sort can enhance the sparse recovery performance.

\begin{figure}[hbtp]
  \centering
  \subfigure[]{
    \label{fig-recovery:a} %% label for first subfigure
    \psfrag{k}[c]{\footnotesize sparsity $s$}
    \psfrag{p}[c][BI][1][0]{\footnotesize percentage}
    \includegraphics[width=2.2in]{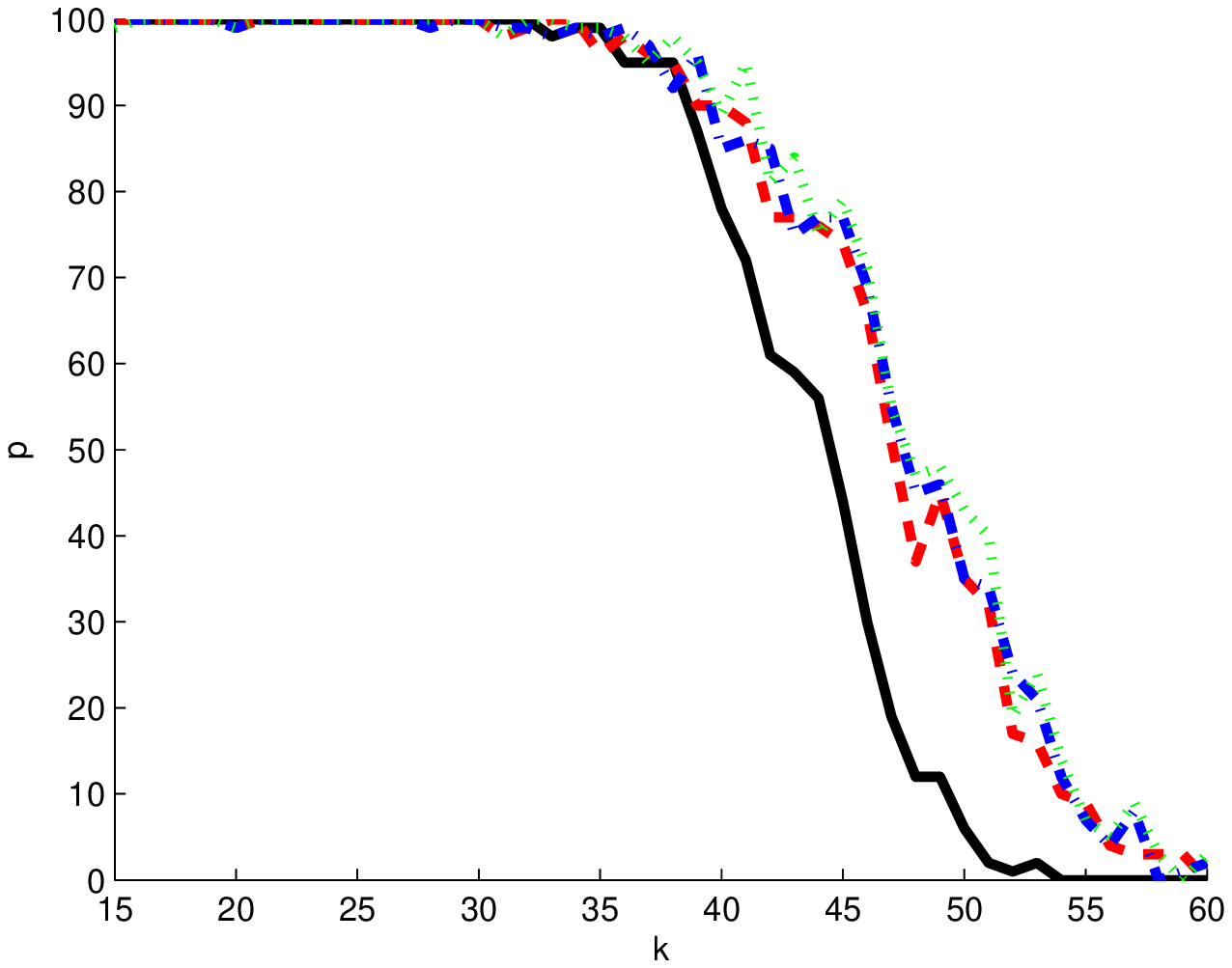}} \quad
  \subfigure[]{
    \label{fig-recovery:b} %% label for first subfigure
    \psfrag{k}[c]{\footnotesize sparsity $s$}
    \psfrag{t}[c][BI][1][0]{\footnotesize time (s)}
    \includegraphics[width=2.2in]{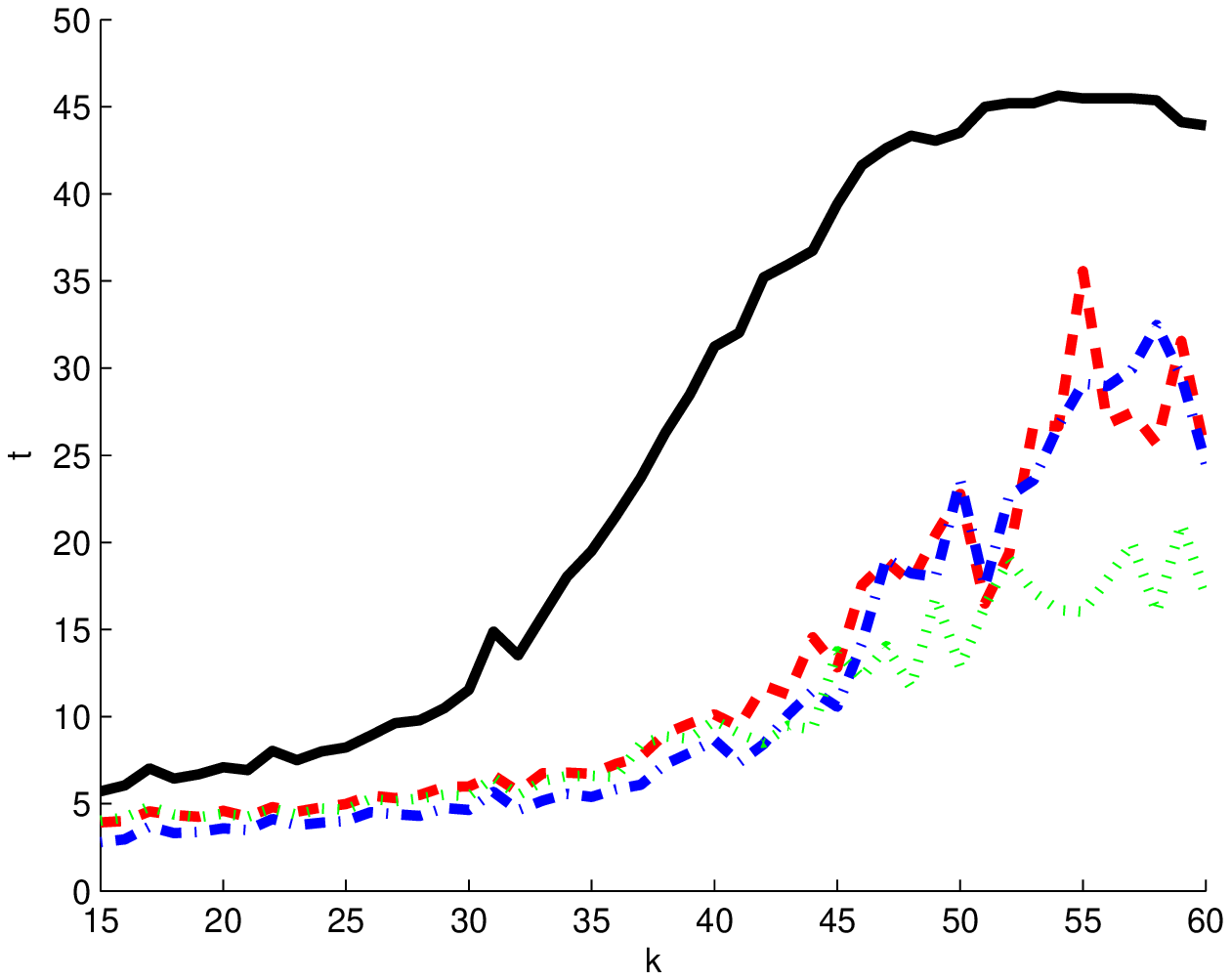}}
    \caption{Performance on signal recovery of ISD (red dashed line), 2level (blue dot-dashed line), mlevel (green dotted line), and IRL1 (black solid line): (a) recovery percentage; (b) average computational time.}\label{fig-recovery}
\end{figure}

Besides the above noise-free experiments, the algorithms are also tested on real-life electrocardiography (ECG) data. The ECG data come from the National Metrology Institute of Germany, which is online available in the PhysioNet \cite{goldberger2000physiobank} \cite{moody2001physionet}. This data set has $15$ signal channels and each channel contains 38400 data points. Notice that ECG signal is not sparse in the time domain and is sparse on the orthogonal Daubechies wavelets (db 10), of which the matrix is denoted by $\mathbf{\Psi}$. Then we start from the first $1024$ data, denoted by $\vu$ and randomly generate one Gaussian matrix $\vA \in \vR^{m \times n}$, where $n = 1024$ and $m$ varies from $64$ to $1024$. Since $\mathbf{\Psi} \vu$ is sparse, we can apply the considered algorithms to recover the signal from $\vb = \vA \mathbf{\Psi} \vu$. We calculate the mean of squared error between the recovered and the original signals, then we move to the next $1024$ data.

In this experiment, ISD, 2level, mlevel, and IRL1 are evaluated. For 2level and mlevel, on the one hand, we can use Algorithm 1 and apply YALL1 to solve the involved weighted $\ell_1$ minimization problems. On the other hand, we can also use the iterative sorted thresholding, i.e., Algorithm 2, to solve the unconstrained problems and evaluate their performance. Not like Algorithm 1,  Algorithm 2 does not calculate $\vu^0$ and hence we set $K$ heuristically as $\lfloor n/5 \rfloor$.

For $m = 128, 256, 512$ and different algorithms, the mean squared error (MSE)  and the corresponding mean computational time are reported in Tables \ref{table-ECG} and \ref{table-ECG-3}, where the best ones in the view of MSE are underlined.
From the results, one can see that solving nonconvex sorted $\ell_1$ minimization by Algorithm 1 provides accurate signals. When the signal length $m$ increases, the iterative sorted thresholding becomes attractive due to its computational effectiveness, because only thresholding and matrix multiplication operations are involved. Compared with weighting by value, weighting by sort shows better performance on this experiment. Especially, when the compression ratio is high, the advantage of mlevel and 2level is significant. In that case, it is worthy designing a flexible and suitable weighting strategy. While, for the low compression ratio situation, we suggest ISD or 2level/mlevel solved by Algorithm 2.

\begin{table}[hbtp]
\begin{center}
\scriptsize
\caption{Mean Squared Error and Computational Time on ECG Data Sets}\label{table-ECG}
\vskip 0.1cm
\begin{tabular}{rrllll}
\toprule
 &     &        &       & ~~~two-level    & ~~multi-level   \\
No. & $m$ &  IRL1  &  ~ISD    & Alg.1 \; Alg.2& Alg.1 \; Alg.2 \\
\midrule
1 %& 64  & 8.107  & 8.914 & 8.374  \quad 8.184 & 8.682 \quad \underline{8.076}\\
  %&     & 0.24 s & 0.83 s& 1.50 s \quad 1.22 s& 2.30 s\quad 1.35 s\\
  & 128 & 4.805  & 3.429 & 3.493  \quad 4.147 & \underline{3.290} \quad 4.006\\
  &     & 0.23 s & 0.73 s& 1.22 s \quad 1.64 s& 1.98 s\quad 1.80 s\\
  & 256 & 1.791  & 1.295 & 1.302  \quad 1.393 & \underline{1.283} \quad 1.394\\
  &     & 0.35 s & 0.73 s& 1.15 s \quad 1.23 s& 2.38 s\quad 1.28 s\\
  & 512 & 0.935  & 0.836 & \underline{0.830}  \quad 0.857 & \underline{0.830} \quad 0.853\\
  &     & 1.09 s & 2.09 s& 3.33 s \quad 1.31 s& 3.45 s\quad 1.33 s\\
\midrule
2 %& 64  & 11.58  & 13.94 & 12.88  \quad 11.59 & 13.68 \quad \underline{11.56}\\
  %&     & 0.30 s & 0.88 s& 1.75 s \quad 1.21 s& 2.78 s\quad 1.29 s\\
  & 128 & 6.205  & 3.189 & 3.072  \quad 6.207 & \underline{2.673} \quad 6.033\\
  &     & 0.24 s & 0.55 s& 0.98 s \quad 1.76 s& 1.90 s\quad 1.85 s\\
  & 256 & 1.407  & 0.925 & 0.922  \quad 1.131 & \underline{0.903} \quad 1.130\\
  &     & 0.31 s & 0.52 s& 0.93 s \quad 1.37 s& 1.89 s\quad 1.41 s\\
  & 512 & 0.673  & 0.596 & 0.594  \quad 0.610 & \underline{0.587} \quad 0.609\\
  &     & 1.09 s & 1.18 s& 2.11 s \quad 1.63 s& 2.11 s\quad 1.63 s\\
\midrule
3 %& 64  & \underline{11.21}  & 12.39 & 11.35  \quad 11.31 & 11.92 \quad 11.62\\
  %&     & 0.37 s & 0.73 s& 1.37 s \quad 1.90 s& 2.12 s\quad 1.05 s\\
  & 128 & 6.188  & 3.897 & 3.692  \quad 6.482 & \underline{3.191} \quad 6.281\\
  &     & 0.35 s & 0.57 s& 1.08 s \quad 2.11 s& 1.95 s\quad 1.29 s\\
  & 256 & 1.732  & 1.085 & 1.128  \quad 1.412 & \underline{1.083} \quad 1.418\\
  &     & 0.50 s & 0.71 s& 1.20 s \quad 1.92 s& 2.37 s\quad 1.39 s\\
  & 512 & 0.756  & 0.678 & 0.682  \quad 0.688 & \underline{0.676} \quad 0.686\\
  &     & 1.24 s & 1.38 s& 2.37 s \quad 1.58 s& 3.01 s\quad 1.42 s\\
\midrule
4 %& 64  & \underline{8.032}  & 9.285 & 8.859  \quad 8.194 & 9.347 \quad 8.218\\
  %&     & 0.33 s & 0.74 s& 1.38 s \quad 1.08 s& 2.15 s\quad 1.28 s\\
  & 128 & 5.598  & 4.939 & \underline{4.689}  \quad 5.335 & 4.828 \quad 5.303\\
  &     & 0.43 s & 0.60 s& 1.11 s \quad 1.01 s& 1.69 s\quad 1.24 s\\
  & 256 & 1.367  & 0.952 & 0.972  \quad 1.096 & \underline{0.949} \quad 1.103\\
  &     & 0.39 s & 0.64 s& 0.84 s \quad 0.96 s& 1.76 s\quad 0.90 s\\
  & 512 & 0.707  & 0.631 & 0.629  \quad 0.646 & \underline{0.624} \quad 0.643\\
  &     & 1.20 s & 1.95 s& 3.11 s \quad 1.19 s& 4.14 s\quad 1.20 s\\
\midrule
5 %& 64  & \underline{7.981}  & 8.932 & 8.250  \quad 8.109 & 8.691 \quad 8.033\\
  %&     & 0.24 s & 0.85 s& 1.51 s \quad 1.16 s& 2.56 s\quad 1.22 s\\
  & 128 & 4.598  & 3.099 & 3.027  \quad 3.288 & \underline{2.821} \quad 3.151\\
  &     & 0.28 s & 0.65 s& 1.21 s \quad 1.16 s& 2.08 s\quad 1.54 s\\
  & 256 & 1.628  & 1.120 & 1.133  \quad 1.311 & \underline{1.111} \quad 1.306\\
  &     & 0.32 s & 0.77 s& 1.33 s \quad 1.20 s& 2.39 s\quad 1.20 s\\
  & 512 & 0.787  & 0.704 & \underline{0.700}  \quad 0.720 & \underline{0.700} \quad 0.718\\
  &     & 1.05 s & 1.36 s& 2.39 s \quad 1.21 s& 4.04 s\quad 1.23 s\\
\midrule
6 %& 64  & \underline{10.55}  & 11.72 & 10.97  \quad 10.60 & 11.92 \quad 10.45\\
  %&     & 0.25 s & 1.00 s& 1.01 s \quad 0.97 s& 1.60 s\quad 1.07 s\\
  & 128 & 5.780  & 3.229 & 2.985  \quad 4.626 & \underline{2.770} \quad 3.757\\
  &     & 0.25 s & 0.59 s& 1.05 s \quad 1.27 s& 1.38 s\quad 1.27 s\\
  & 256 & 1.243  & 0.791 & 0.821  \quad 0.988 & \underline{0.787} \quad 0.995\\
  &     & 0.32 s & 0.51 s& 0.86 s \quad 0.89 s& 2.00 s\quad 0.95 s\\
  & 512 & 0.559  & 0.504 & 0.506  \quad 0.516 & \underline{0.500} \quad 0.515\\
  &     & 0.87 s & 1.12 s& 1.92 s \quad 1.00 s& 3.15 s\quad 1.02 s\\
\midrule
7 %& 64  & 9.463  & 9.805 & \underline{9.085}  \quad 10.62 & 9.559 \quad 10.46\\
  %&     & 0.21 s & 0.71 s& 1.22 s \quad 1.21 s& 2.04 s\quad 1.25 s\\
  & 128 & 5.376  & 3.666 & 3.715  \quad 4.889 & \underline{3.390} \quad 4.712\\
  &     & 0.18 s & 0.60 s& 0.91 s \quad 1.22 s& 1.55 s\quad 1.20 s\\
  & 256 & 1.181  & \underline{0.685} & 0.730  \quad 1.047 & 0.707 \quad 1.054\\
  &     & 0.20 s & 0.42 s& 0.68 s \quad 1.00 s& 1.32 s\quad 1.09 s\\
  & 512 & 0.410  & 0.395 & 0.401  \quad \underline{0.391} & 0.396 \quad 0.392\\
  &     & 0.42 s & 0.49 s& 1.84 s \quad 0.95 s& 1.39 s\quad 0.95 s\\
\bottomrule
\end{tabular}
\end{center}
\end{table}

\begin{table}[hbtp]
\begin{center}
\scriptsize
\caption{Mean Squared Error and Computational Time on ECG Data Sets (cont.)}\label{table-ECG-3}
\vskip 0.1cm
\begin{tabular}{rrllll}
\toprule
 &     &        &       & ~~~two-level    & ~~multi-level   \\
No. & $m$ &  IRL1  &  ~ISD    & Alg.1 \; Alg.2& Alg.1 \; Alg.2 \\
\midrule
8 %& 64  & \underline{12.26}  & 14.09 & 12.73  \quad 12.53 & 13.45 \quad 12.44\\
  %&     & 0.23 s & 0.87 s& 1.69 s \quad 1.30 s& 1.73 s\quad 1.43 s\\
  & 128 & 7.397  & 3.993 & 4.302  \quad 5.903 & \underline{3.652} \quad 5.612\\
  &     & 0.28 s & 0.65 s& 1.23 s \quad 1.70 s& 1.97 s\quad 1.04 s\\
  & 256 & 1.219  & \underline{0.596} & 0.652  \quad 0.810 & 0.613 \quad 0.890\\
  &     & 0.37 s & 0.84 s& 1.29 s \quad 1.79 s& 2.32 s\quad 1.38 s\\
  & 512 & 0.414  & 0.379 & 0.388  \quad 0.386 & \underline{0.370} \quad 0.386\\
  &     & 0.75 s & 0.72 s& 2.25 s \quad 1.51 s& 2.12 s\quad 1.08 s\\
\midrule
9 %& 64  & \underline{15.46}  & 17.18 & 15.66  \quad 15.91 & 16.67 \quad 15.82\\
  %&     & 0.23 s & 0.95 s& 1.55 s \quad 1.46 s& 2.49 s\quad 1.52 s\\
  & 128 & 8.459  & 4.267 & 4.509  \quad 6.982 & \underline{3.859} \quad 8.727\\
  &     & 0.21 s & 0.49 s& 0.96 s \quad 1.60 s& 1.67 s\quad 1.67 s\\
  & 256 & 1.329  & \underline{0.685} & 0.726  \quad 0.832 & 0.698 \quad 0.853\\
  &     & 0.31 s & 0.74 s& 1.20 s \quad 1.89 s& 2.03 s\quad 1.91 s\\
  & 512 & 0.429  & 0.416 & 0.417  \quad 0.407 & \underline{0.400} \quad 0.408\\
  &     & 0.49 s & 0.63 s& 0.79 s \quad 1.52 s& 1.31 s\quad 1.55 s\\
\midrule
10%& 64  & 8.576  & 8.020 & \underline{7.570}  \quad 9.067 & 8.187 \quad 8.712\\
  %&     & 0.26 s & 0.76 s& 1.26 s \quad 1.24 s& 2.36 s\quad 1.23 s\\
  & 128 & 4.685  & 3.390 & 3.287  \quad 5.151 & \underline{3.212} \quad 3.888\\
  &     & 0.25 s & 0.86 s& 1.53 s \quad 1.57 s& 2.38 s\quad 1.56 s\\
  & 256 & 1.154  & 0.635 & 0.654  \quad 0.665 & \underline{0.633} \quad 0.657\\
  &     & 0.37 s & 0.72 s& 1.28 s \quad 1.81 s& 2.36 s\quad 1.57 s\\
  & 512 & 0.427  & 0.386 & 0.392  \quad 0.399 & \underline{0.381} \quad 0.398\\
  &     & 0.60 s & 0.72 s& 1.23 s \quad 1.14 s& 2.30 s\quad 1.14 s\\
\midrule
11%& 64  & \underline{5.991}  & 6.552 & 6.095  \quad 6.178 & 6.471 \quad 6.094\\
  %&     & 0.22 s & 0.69 s& 1.25 s \quad 0.96 s& 2.05 s\quad 1.02 s\\
  & 128 & 3.522  & 2.448 & 2.064  \quad 3.455 & \underline{1.937} \quad 3.281\\
  &     & 0.20 s & 0.50 s& 0.95 s \quad 1.06 s& 1.52 s\quad 1.12 s\\
  & 256 & 0.930  & \underline{0.588} & 0.615  \quad 0.728 & 0.583 \quad 0.729\\
  &     & 0.28 s & 0.56 s& 1.28 s \quad 0.77 s& 1.87 s\quad 0.79 s\\
  & 512 & 0.402  & 0.353 & 0.359  \quad 0.370 & \underline{0.351} \quad 0.358\\
  &     & 0.75 s & 1.12 s& 2.02 s \quad 0.70 s& 3.94 s\quad 0.69 s\\
\midrule
12%& 64  & \underline{5.356}  & 6.274 & 5.960  \quad 5.473 & 6.270 \quad 5.242\\
  %&     & 0.22 s & 0.56 s& 1.08 s \quad 0.85 s& 1.68 s\quad 0.98 s\\
  & 128 & 3.159  & \underline{1.776} & 1.921  \quad 2.471 & 1.820 \quad 2.386\\
  &     & 0.22 s & 0.53 s& 0.91 s \quad 0.94 s& 1.48 s\quad 0.95 s\\
  & 256 & 0.775  & 0.502 & 0.503  \quad 0.576 & \underline{0.491} \quad 0.577\\
  &     & 0.28 s & 0.44 s& 0.79 s \quad 0.56 s& 1.65 s\quad 0.57 s\\
  & 512 & 0.348  & 0.308 & 0.310  \quad 0.322 & \underline{0.306} \quad 0.321\\
  &     & 0.78 s & 1.38 s& 2.17 s \quad 0.58 s& 3.23 s\quad 0.59 s\\
\midrule
13%& 64  & 4.207  & 4.016 & 3.868  \quad 4.333 & \underline{3.901} \quad 4.181\\
  %&     & 0.26 s & 0.60 s& 1.03 s \quad 0.87 s& 1.83 s\quad 0.95 s\\
  & 128 & 2.316  & 1.614 & 1.448  \quad 2.281 & \underline{1.417} \quad 2.222\\
  &     & 0.25 s & 0.58 s& 0.95 s \quad 0.84 s& 1.64 s\quad 0.90 s\\
  & 256 & 0.514  & \underline{0.304} & 0.320  \quad 0.424 & 0.310 \quad 0.427\\
  &     & 0.37 s & 0.67 s& 1.25 s \quad 0.72 s& 1.94 s\quad 0.70 s\\
  & 512 & 0.159  & 0.158 & 0.160  \quad 0.168 & \underline{0.156} \quad 0.187\\
  &     & 0.60 s & 0.55 s& 0.95 s \quad 0.54 s& 1.62 s\quad 0.55 s\\
\midrule
14%& 64  & \underline{7.343}  & 8.305 & 8.070  \quad 7.599 & 8.486 \quad 7.719\\
  %&     & 0.20 s & 0.54 s& 1.06 s \quad 0.94 s& 1.79 s\quad 1.08 s\\
  & 128 & 4.169  & 2.558 & 2.344  \quad 2.621 & \underline{2.223} \quad 3.417\\
  &     & 0.21 s & 0.44 s& 0.81 s \quad 1.20 s& 1.35 s\quad 1.25 s\\
  & 256 & 0.958  & 0.568 & 0.588  \quad 0.647 & \underline{0.566} \quad 0.750\\
  &     & 0.27 s & 0.45 s& 0.76 s \quad 0.73 s& 1.51 s\quad 0.75 s\\
  & 512 & 0.410  & 0.355 & 0.359  \quad 0.377 & \underline{0.353} \quad 0.377\\
  &     & 0.70 s & 0.87 s& 1.59 s \quad 0.70 s& 3.26 s\quad 0.71 s\\
\midrule
15%& 64  & \underline{5.006}  & 5.293 & 5.088  \quad 5.114 & 5.254 \quad 5.057\\
  %&     & 0.22 s & 0.82 s& 1.57 s \quad 0.95 s& 2.56 s\quad 1.07 s\\
  & 128 & 2.836  & \underline{1.754} & 1.899  \quad 2.488 & 1.791 \quad 2.291\\
  &     & 0.20 s & 0.40 s& 0.73 s \quad 1.01 s& 1.33 s\quad 1.07 s\\
  & 256 & 0.630  & 0.342 & 0.359  \quad 0.413 & \underline{0.337} \quad 0.411\\
  &     & 0.27 s & 0.52 s& 0.83 s \quad 0.61 s& 1.54 s\quad 0.61 s\\
  & 512 & 0.175  & \underline{0.174} & 0.178  \quad 0.201 & \underline{0.174} \quad 0.178\\
  &     & 0.48 s & 0.43 s& 0.79 s \quad 0.52 s& 1.32 s\quad 0.52 s\\
\bottomrule
\end{tabular}
\end{center}
\end{table}

\section{Conclusion}

The nonconvex sorted $\ell_1$ minimization is proposed to enhance the sparse signal recovery. In this penalty, the set of the weights is fixed and the weights are assigned based on the ranks of the corresponding components among all the components in magnitude. Iteratively reweighted $\ell_1$ minimization and iterative sorted thresholding are proposed to solve nonconvex sorted $\ell_1$ minimization problems. Both methods are generalizations of existing methods. Iteratively reweighted $\ell_1$ minimization is a generalization of iterative support detection, and iterative sorted thresholding is a genelarization of iterative hard thresholding. Both methods are shown to converge to a local optimum. The numerical experiments demonstrate the better performance of assigning weighted by sort.

%\begin{acknowledgements}
%If you'd like to thank anyone, place your comments here and remove the percent signs.
%\end{acknowledgements}

\bibliographystyle{spmpsci}      % mathematics and physical sciences
\bibliography{mcref,nonconvex}   % name your BibTeX data base
\end{document}